\documentclass[11pt,a4paper]{article}%
\usepackage{amssymb,amsmath,amsfonts,amsthm,array,bm,color}%
\usepackage{amsmath}%
\usepackage{amsfonts}%
\usepackage{amssymb}%
\usepackage{graphicx}
\providecommand{\U}[1]{\protect\rule{.1in}{.1in}}
\newtheorem{theorem}{Theorem}[section]
\newtheorem{lemma}[theorem]{Lemma}
\newtheorem{corollary}[theorem]{Corollary}
\newtheorem{proposition}[theorem]{Proposition}
\newtheorem{remark}[theorem]{Remark}

\newtheorem{definition}[theorem]{Definition}
\newtheorem{conjecture}[theorem]{Conjecture}

\def\<{\langle}
\def\>{\rangle}
\def\d{{\rm d}}

\def\div{{\rm div}}
\def\E{\mathbb{E}}
\def\N{\mathbb{N}}
\def\P{\mathbb{P}}
\def\R{\mathbb{R}}
\def\T{\mathbb{T}}
\def\Z{\mathbb{Z}}
\def\Leb{{\rm Leb}}

\begin{document}

\title{$\rho$-white noise solution to 2D stochastic Euler equations}
\author{Franco Flandoli\footnote{Email: flandoli@dma.unipi.it. Dipartimento di Matematica, Universit\`{a} di Pisa, Largo Bruno Pontecorvo 5, Pisa, Italy.} \ and
Dejun Luo\footnote{Email: luodj@amss.ac.cn. RCSDS, Academy of Mathematics and Systems Science, Chinese Academy of Sciences, Beijing 100190, China, and School of Mathematical Sciences, University of the Chinese Academy of Sciences, Beijing 100049, China. }}

\maketitle

%\maketitle
\makeatletter
%'@' is now a normal "letter" for TeX
\renewcommand\theequation{\thesection.\arabic{equation}}
\@addtoreset{equation}{section} \makeatother
%'@' is restored as a "non-letter" character for TeX

\begin{abstract}
A stochastic version of 2D Euler equations with transport type
noise in the vorticity is considered, in the framework of Albeverio--Cruzeiro
theory \cite{AC} where the equation is considered with random initial conditions
related to the so called enstrophy measure. The equation is studied by an
approximation scheme based on random point vortices. Stochastic processes
solving the Euler equations are constructed and their density with respect
to the enstrophy measure is proved to satisfy a continuity equation in weak
form. Relevant in comparison with the case without noise is the fact that
here we prove a gradient type estimate for the density. Although we
cannot prove uniqueness for the continuity equation, we discuss how the
gradient type estimate may be related to this open problem.
\end{abstract}

\section{Introduction}

This work is devoted to the investigation of 2D Euler equations with a
Gaussian distributed initial condition and perturbed by multiplicative noise
in transport form. Besides its intrinsic interest as a model of stochastic
fluid mechanics, this topic lies at the intersection of several research lines
of recent interest, a fact that was our main motivation. On one side, relevant
classes of PDEs, of dispersive type, have been solved recently in spaces of
low regularity, replacing arbitrary initial conditions by almost every
initial condition with respect to a suitable measure, see \cite{Tzvetkov} for
a review. Solvability of deterministic equations in infinite dimensional
spaces in a probabilistic sense with respect to initial conditions has also been
approached by means of the associated infinite dimensional continuity
equation, see for instance \cite{Cruz, Bogach, AmbrFigalli, FangLuo, AmbrTrevisan, DFR, KR}. On the other side,
multiplicative transport noise has been proven to regularize certain singular PDEs,
see the review \cite{Flandoli11}; in particular, related to the present work, it
regularizes the dynamics of Euler point vortices, which is well posed in the
deterministic case only for almost every initial configuration with respect to
Lebesgue measure, while it is for all initial conditions when a suitable noise
is added to Euler equations, see \cite{FGP} and \cite{DFV} for a related result on Vlasov--Poisson equations.
That a suitable transport noise regularizes 2D Euler equations is an open
problem, see \cite{FlaCIB}. The approach presented here does not solve this
question yet but poses the basis for further investigations on this
regularization by noise question, due to the gradient type estimates on the
density. In particular, in Theorem \ref{thm-integrability} we
investigate a key property in the direction of uniqueness and, from the
assumptions of that theorem, we identify a new example of transport type
noise, at the border of the regularity class considered in this paper, that
requires to be studied in future researches. Let us now describe in detail the contribution of the present paper
to the previous range of topics.

Let $\T^2=\R^2/ \Z^2$ be the two dimensional torus. The two dimensional Euler equations in vorticity form reads as
  \begin{equation}\label{Euler-vorticity}
  \partial_t\omega_t+ u_t\cdot \nabla \omega_t=0,\quad \omega|_{t=0}= \omega_0,
  \end{equation}
where $u_t=(u^1_t, u^2_t)$ is the divergence free velocity field and $\omega_t =\partial_2 u^1_t- \partial_1 u^2_t$ is the vorticity. We refer the reader to the introduction of \cite{Flandoli} for a list of well posedness results on \eqref{Euler-vorticity} under different regularity assumptions on $\omega_0$.

We consider the equation \eqref{Euler-vorticity} perturbed by random noises:
  \begin{equation}\label{stoch-Euler-vorticity}
  \d \omega_t+ u_t\cdot \nabla \omega_t\, \d t +\sum_{j=1}^\infty \sigma_j\cdot \nabla\omega_t\circ \d W^j_t=0,
  \end{equation}
where $\{\sigma_j: j\in \N\}$ and $\big\{(W^j_t)_{t\geq 0}: j\in \N \big\}$ are, respectively, a family of smooth divergence free vector fields on $\T^2$ and a family of independent real Brownian motions defined on a filtered probability space $(\Theta, \mathcal F, (\mathcal F_t), \P)$. The weak formulation of \eqref{stoch-Euler-vorticity} is
  \begin{equation}\label{weak-Euler-vorticity}
  \d \<\omega_t,\phi\>= \<\omega_t, u_t\cdot \nabla \phi\>\, \d t +\sum_{j=1}^\infty \<\omega_t,\sigma_j\cdot \nabla\phi\> \circ \d W^j_t,
  \end{equation}
where $\phi\in C^\infty (\T^2)$ and $\<\, ,\>$ is the duality between the space $C^\infty (\T^2)'$ of distributions and $C^\infty (\T^2)$. The It\^o form of the above equation is given by
  \begin{equation*}
  \d \<\omega_t,\phi\>= \<\omega_t, u_t\cdot \nabla \phi\>\, \d t +\sum_{j=1}^\infty \<\omega_t,\sigma_j\cdot \nabla\phi\> \,\d W^j_t + \frac12 \sum_{j=1}^\infty \big\<\omega_t,\sigma_j \cdot \nabla (\sigma_j \cdot \nabla \phi) \big\> \,\d t.
  \end{equation*}
This equation can be rewritten in the weak vorticity formulation by using the Biot--Savart kernel $K(x-y)$ on the torus. It is known that (see \cite{Schochet}) $K$ is smooth for $x\neq y$, $K(y-x)= - K(x-y)$ and $|K(x-y)|\leq C/|x-y|$ for $|x-y|$ small enough. By the Biot--Savart law,
  $$u_t(x)= \int_{\T^2} K(x-y)\, \omega_t(\d y).$$
Therefore,
  $$\<\omega_t, u_t\cdot \nabla \phi\>= \int_{\T^2}\int_{\T^2} K(x-y)\cdot \nabla\phi(x)\, \omega_t(\d y)\, \omega_t(\d x).$$
Since $K(y-x)= - K(x-y)$, we can rewrite the above quantity in the symmetric form:
  $$\<\omega_t, u_t\cdot \nabla \phi\>= \int_{\T^2}\int_{\T^2} H_\phi(x,y)\, \omega_t(\d y)\, \omega_t(\d x) = \<\omega_t \otimes \omega_t, H_\phi\>,$$
where
  $$H_\phi(x,y)=\frac12 K(x-y) \cdot (\nabla\phi(x) -\nabla\phi(y)).$$
Now we obtain the weak vorticity formulation of the 2D stochastic Euler equation:
  \begin{equation}\label{weak-Euler-vorticity-1}
  \d \<\omega_t,\phi\>= \<\omega_t \otimes \omega_t, H_\phi\> \, \d t +\sum_{j=1}^\infty \<\omega_t,\sigma_j\cdot \nabla\phi\> \,\d W^j_t + \frac12 \sum_{j=1}^\infty \big\<\omega_t,\sigma_j \cdot \nabla (\sigma_j \cdot \nabla \phi) \big\> \,\d t.
  \end{equation}

We need some notations in order to introduce the notion of solution to \eqref{weak-Euler-vorticity-1}. For any $s\in \R$, we write $H^s(\T^2)$ for the usual Sobolev space on $\T^2$, and $H^{-1-}(\T^2) =\cap_{\delta>0} H^{-1-\delta} (\T^2)$. Let $\omega_{WN}: \Theta\to C^\infty (\T^2)'$ be the white noise on $\T^2$, which is by definition a Gaussian random distribution such that
  $$\E \big[\<\omega_{WN}, \phi\> \<\omega_{WN}, \psi\> \big]= \<\phi, \psi\>, \quad \mbox{for all } \phi, \psi \in C^\infty (\T^2),$$
where $\<\, ,\>$ on the r.h.s. is the inner product in $L^2(\T^2, \d x)$. The law of the white noise $\omega_{WN}$, called the enstrophy measure and denoted by $\mu$, is supported by $H^{-1-}(\T^2)$. It is proven in \cite[Theorem 8]{Flandoli} that, under the probability measure $\mu$, $\<\omega \otimes \omega, H_\phi\>$ is a square integrable r.v. on $H^{-1-}(\T^2)$.

\begin{definition}[$\rho$-white noise solution]\label{def-solution}
Let $\rho: [0,T]\times H^{-1-}(\T^2)\to [0,\infty)$ satisfy $\int \rho_t^q \,\d\mu \leq C$ for some constants $C>0,\, q>1$, and $\int \rho_t \,\d\mu =1$ for every $t\in [0,T]$. We say that a measurable map $\omega: \Theta\times [0,T] \to C^\infty(\T^2)'$, which has trajectories of class $C\big([0,T], H^{-1-}(\T^2) \big)$ and is adapted to $(\mathcal F_t)_{t\geq 0}$, is a $\rho$-white noise solution of the stochastic Euler equations \eqref{stoch-Euler-vorticity} if $\omega_t$ has law $\rho_t \mu$ at every time $t\in [0,T]$, and for every $\phi\in C^\infty(\T^2)$, the following identity holds $\P$-a.s., uniformly in time,
  \begin{equation}\label{def-solution-1}
  \aligned
  \<\omega_t,\phi\> &= \<\omega_0,\phi\> +\int_0^t \<\omega_s\otimes \omega_s, H_\phi\>\, \d s +\sum_{j=1}^\infty \int_0^t \<\omega_s,\sigma_j \cdot \nabla \phi\>\,\d W^j_s\\
  &\hskip13pt + \frac12 \sum_{j=1}^\infty \int_0^t \big\<\omega_s,\sigma_j \cdot \nabla (\sigma_j \cdot \nabla \phi) \big\>\,\d s.
  \endaligned
  \end{equation}
\end{definition}

Before presenting the main results of this paper, we introduce our assumptions on the vector fields $\{\sigma_j:j\in \N\}$:
\begin{itemize}
\item[\textbf{(H1)}] For all $j\in \N$, the vector fields $\sigma_j$ are periodic, smooth and $\div(\sigma_j)=0$.
\item[\textbf{(H2)}] The series $\sum_{j=1}^\infty \|\sigma_j\|_\infty^2$ and $\sum_{j=1}^\infty \|\sigma_j \cdot \nabla \sigma_j \|_\infty$ are convergent.
\end{itemize}

\begin{remark}\label{1-rem-1}
Under the conditions of Definition \ref{def-solution}, we have $\<\omega_s\otimes \omega_s, H_\phi\> \in L^1\big( \Theta, L^1([0,T]) \big)$ by \cite[Theorem 15]{Flandoli}. Moreover, we deduce from {\rm \textbf{(H2)}} that the martingale part in \eqref{def-solution-1} is a square integrable martingale. Indeed, since $\omega_s$ is distributed as $\rho_s \mu$, by H\"older's inequality,
  $$\E \big(\<\omega_s,\sigma_j \cdot \nabla \phi\>^2 \big) = \E_\mu \big(\rho_s\, \<\omega,\sigma_j \cdot \nabla \phi\>^2 \big) \leq \big(\E_\mu \rho_s^q \big)^{1/q} \big(\E_\mu \<\omega,\sigma_j \cdot \nabla \phi\>^{2q'}\big)^{1/q'},$$
where $\E_\mu$ denotes the expectation on $H^{-1-}$ w.r.t. the enstrophy measure $\mu$. Recall that if $\xi\sim N(0,\sigma^2)$, then for any $p>1$, one has $\E (|\xi|^p) \leq C_p\, \sigma^p$ for some constant $C_p>0$. Under $\mu$, $\<\omega,\sigma_j \cdot \nabla \phi\>$ is a centered Gaussian r.v. with variance $\int_{\T^2} |\sigma_j \cdot \nabla \phi|^2 \,\d x \leq \|\sigma_j\|_\infty^2 \|\nabla\phi \|_\infty^2$. Combining these facts with the property of $\rho_s$ yields
  $$\E \big(\<\omega_s,\sigma_j \cdot \nabla \phi\>^2 \big) \leq C^{1/q} C_{q'} \|\sigma_j\|_\infty^2 \|\nabla\phi \|_\infty^2 . $$
This together with {\rm \textbf{(H2)}} gives us
  $$\sum_{j=1}^\infty \int_0^t \E \big(\<\omega_s,\sigma_j \cdot \nabla \phi\>^2 \big) \,\d s \leq C_q t\|\nabla\phi \|_\infty^2 \sum_{j=1}^\infty \|\sigma_j\|_\infty^2 < \infty,$$
which implies the claim. In the same way, one can show that
  $$\sum_{j=1}^\infty \int_0^t \E \big|\big\<\omega_s,\sigma_j \cdot \nabla (\sigma_j \cdot \nabla \phi) \big\> \big|\,\d s \leq C_q t \sum_{j=1}^\infty \big(\|\nabla^2\phi \|_\infty \|\sigma_j\|_\infty^2 + \|\nabla\phi \|_\infty \|\sigma_j \cdot \nabla \sigma_j \|_\infty \big) <\infty. $$
\end{remark}

Now we can present the first main result.

\begin{theorem}[Existence] \label{thm-main-result}
Given $\rho_0 \in C_b \big( H^{-1-}(\T^2)\big)$ such that $\rho_0 \geq 0$ and $\int \rho_0\,\d\mu=1$. Under the assumptions {\rm \textbf{(H1)}} and {\rm \textbf{(H2)}}, there exist a bounded measurable function $\rho:[0,T] \times H^{-1-}(\T^2) \to [0, \|\rho_0\|_\infty]$, and a filtered probability space $(\Theta, \mathcal F, (\mathcal F_t), \P)$ on which there are defined a $(\mathcal F_t)$-adapted process $\omega_\cdot: \Theta \times [0,T]\to C^\infty (\T^2)'$ and a sequence of $(\mathcal F_t)$-Brownian motions $\{(W^j_t)_{t\geq 0}: j\in \N\}$, such that $\omega_\cdot$ is a $\rho$-white noise solution of the stochastic Euler equation \eqref{stoch-Euler-vorticity}.
\end{theorem}

Our next result is concerned with the regularity properties of the density $\rho_t$, for which we need some more notations. Given two elements $\omega,\eta\in C^{\infty} ( \T^2)^{\prime}$ and a function $G:C^{\infty} (  \mathbb{T}^{2})^{\prime} \to\mathbb{R}$, we write $\langle \eta,D_{\omega} G (\omega) \rangle $ for the limit
  \[
  \left\langle \eta,D_{\omega}G\left( \omega\right)  \right\rangle
  =\lim_{\varepsilon\to 0}\frac{G( \omega+\varepsilon \eta) -G( \omega) }{\varepsilon}
  \]
when it exists. For instance, if $G$ is taken from
  $$\aligned
  \mathcal{FC}_P= \big\{G:C^{\infty} (\T^{2})^{\prime} \to\mathbb{R}\, \big| & G(\omega)=g( \langle \omega,\phi_{1}\rangle ,\ldots ,\langle \omega,\phi_{n}\rangle ) \mbox{ for some } n\in \N\\
  & \mbox{and } g\in C_P^\infty(\R^n),\phi_{1},\ldots ,\phi_{n}\in C^\infty(\T^2) \big\},
  \endaligned$$
where $C_P^\infty(\R^n)$ is the space of smooth functions on $\R^n$ having polynomial growth together with all the derivatives, then
  \[
  \langle \eta,D_{\omega} G (\omega) \rangle =\sum_{j=1}^{n}\partial_{j} g(\langle \omega, \phi_{1}\rangle ,\ldots ,\langle \omega,\phi_{n}\rangle ) \langle \eta, \phi_{j}\rangle .
  \]
We will also write $D_{\omega} G (\omega)= \sum_{j=1}^{n}\partial_{j} g(\langle \omega,\phi_{1}\rangle ,\ldots ,\langle \omega,\phi_{n}\rangle ) \, \phi_{j}$.

For our purpose, we shall need test functions which depend on time. Hence we denote by
  $$\aligned
  \mathcal{FC}_{P,T}= \bigg\{F:[0,T]\times C^{\infty} (\T^{2})^{\prime} \to\R \, \Big| & F(t,\omega) =\sum_{i=1}^m g_i(t) f_i(\omega) \mbox{ for some } m\in \N \\
  &\mbox{and } g_i\in C^1([0,T]),\, f_i\in \mathcal{FC}_P,\, 1\leq i\leq m\bigg\}.
  \endaligned$$
For $F\in \mathcal{FC}_{P,T}$ given by $F(t,\omega) =\sum_{i=1}^m g_i(t) f_i(\<\omega, \phi_1\>, \ldots , \<\omega, \phi_n\>) $, we have
  $$D_\omega F(t,\omega)= \sum_{i=1}^m g_i(t)\sum_{j=1}^n \partial_j f_i(\<\omega, \phi_1\>, \ldots , \<\omega, \phi_n\>)\, \phi_j.$$
Set
  $$\<b(\omega), D_\omega F(t,\omega)\>:= \sum_{i=1}^m g_i(t)\sum_{j=1}^n \partial_j f_i(\<\omega, \phi_1\>, \ldots , \<\omega, \phi_n\>) \big\< \omega\otimes \omega, H_{\phi_j}\big\>,$$
where $\big\< \omega\otimes \omega, H_{\phi_j}\big\>,\, j=1,\ldots, n$, are limits of Cauchy sequences in $L^2\big(H^{-1-}(\T^2), \mu\big)$ (see \cite[Theorem 8]{Flandoli}). Hence $\<b(\omega), D_\omega F(t,\omega)\>$ belongs to $C\big( [0,T], L^r\big(H^{-1-}(\T^2), \mu\big) \big)$ for all $r\in [1,2)$.

\begin{theorem}[Regularity] \label{thm-main-result-2}
Let $\rho:[0,T] \times H^{-1-}(\T^2) \to [0, \|\rho_0\|_\infty]$ be the density function given in Theorem \ref{thm-main-result}.
\begin{itemize}
\item[\rm (i)] For any $F\in \mathcal{FC}_{P,T}$ with $F(T,\cdot)=0$, the function $\rho$ satisfies
  \begin{equation}\label{thm-main-result-2.1}
  \aligned
  0&= \int F(0,\omega) \rho_0(\omega)\mu(\d\omega) + \int_0^T \!\int \big[(\partial_t F)(t, \omega) + \<b(\omega), D_\omega F(t,\omega)\> \big] \rho_t(\omega) \mu(\d\omega) \d t \\
  &\hskip13pt + \frac12 \sum_{k=1}^\infty \int_0^T \!\int \big\< \sigma_k \cdot \nabla \omega, D_\omega \<\sigma_k \cdot \nabla \omega, D_\omega F(t,\omega)\> \big\> \rho_t(\omega) \mu(\d\omega) \d t.
  \endaligned
  \end{equation}

\item[\rm (ii)] For every $k\in \N$, $\big\< \sigma_k\cdot \nabla \omega, D_\omega \rho_t(\omega)\big\>$ exists in the distributional sense and the gradient estimate holds:
  \begin{equation}\label{thm-main-result-2.2}
  \sum_{k=1}^\infty \int_0^T\!\int \big\< \sigma_k\cdot \nabla \omega, D_\omega \rho_t(\omega)\big\>^2 \, \mu( \d\omega) \d t \leq \|\rho_0\|_\infty^2.
  \end{equation}
\end{itemize}
\end{theorem}

\begin{remark}\label{1-rem}
\begin{itemize}
\item[\rm(1)] We briefly explain the meaning of the second order term in \eqref{thm-main-result-2.1}. The distribution $\sigma_{k} \cdot \nabla\omega$ is understood as follows:
  $$\left\langle \sigma_{k}\cdot\nabla\omega, \phi\right\rangle :=-\left\langle \omega,\sigma_{k}\cdot\nabla\phi \right\rangle,\quad \phi\in C^\infty ( \T^2),$$
since we assume $\sigma_{k}$ is smooth and divergence free. Given $G\in \mathcal{FC}_P$ of the form $G( \omega)= g(\<\omega, \phi_1\>, \ldots , \<\omega, \phi_n\>)$, we consider the new functional $H:C^{\infty}( \mathbb{T}^{2} )^{\prime} \to \R$ defined by
  $$H( \omega) = \langle \sigma_{k} \cdot \nabla\omega, D_{\omega} G( \omega) \rangle = - \sum_{j=1}^n \partial_j g(\<\omega, \phi_1\>, \ldots , \<\omega, \phi_n\>) \<\omega, \sigma_k \cdot \nabla \phi_j\>.$$
Then $H\in \mathcal{FC}_P$. In Lemma \ref{sec-4.1-lem} below we compute explicitly the term $\langle \sigma_{k}\cdot\nabla \omega,D_{\omega} H( \omega) \rangle $ (see also Remark \ref{sec-4.1-rem}).

\item[\rm(2)] We explain here what we mean by $\big\< \sigma_k\cdot \nabla \omega, D_\omega \rho_t(\omega)\big\>$ exists in the distributional sense for all $k\in \N$. It comes from the equality \eqref{eq-17} which looks like an integration by parts formula. Thanks to \eqref{eq-17} and the fact that $\div_\mu (\sigma_k\cdot \nabla \omega)=0$ (see Lemma \ref{lem-divergence}), it is natural to define $\big\< \sigma_k\cdot \nabla \omega, D_\omega \rho_t(\omega)\big\>= G_k(t, \omega)$ with some $G\in L^2 \big(\N\times [0,T] \times H^{-1-}, \# \otimes \d t \otimes \mu\big)$, where $\#$ is the counting measure on the set $\N$ of natural numbers.
\end{itemize}
\end{remark}

At the heuristic level, the gradient estimate \eqref{thm-main-result-2.2} can be guessed by an energy-type computation on $\rho_t$, using skew-symmetry with respect to $\mu$ of certain differential operators. However, energy-type computations cannot be performed rigorously on weak solutions satisfying \eqref{thm-main-result-2.1}. Our strategy will be to prove a gradient estimate for the density associated to the point vortex approximation and then pass to the limit.

With the gradient estimate \eqref{thm-main-result-2.2} in hand, it is tempting to prove the uniqueness of the equation \eqref{thm-main-result-2.1}. It turns out that a key property, to prove an uniqueness claim, is to have that the function $\left\langle b\left(  \omega\right)  ,D_{\omega
}\rho_{t}\right\rangle $ should be well defined in a suitable sense and
integrable. After some formal calculations, we find that the drift term $\<b(\omega), D_\omega \rho_t\>$ can be expressed as
  $$\<b(\omega), D_\omega \rho_t\>= \sum_{k=1}^\infty \frac{\<\omega, \sigma_k \ast K\> } {\|\sigma_k\|_{L^2}^2} \<\sigma_k \cdot \nabla \omega, D_\omega \rho_t\>,$$
where $(\sigma_k \ast K)(x)= \int_{\T^2} \sigma_k(x-y) \cdot K(y)\,\d y$ is a smooth function on $\T^2$. Consider the following family of vector fields: for $\gamma\geq 2$,
  \begin{equation}\label{vector-fields}
  \sigma_k(x)= {\rm e}^{2\pi {\rm i} k\cdot x}\frac{k^\perp}{|k|^\gamma}, \quad  x\in \T^2,\, k\in \Z^2_0:= \Z^2\setminus \{0\}.
  \end{equation}
If $\gamma >2$, since $\sigma_k \cdot \nabla \sigma_k = 0$, it is obvious that these vector fields satisfy our assumptions \textbf{(H1)} and \textbf{(H2)}. Using the Fourier expansion of $K$, one has $ (\sigma_k \ast K)(x)= 2\pi {\rm i}\, {\rm e}^{2\pi {\rm i} k\cdot x}/ |k|^\gamma$. Therefore,
  $$\<b(\omega), D_\omega \rho_t\>= 2\pi {\rm i} \sum_{k\in \Z^2_0} |k|^{\gamma -2} \big\<\omega, {\rm e}^{2\pi {\rm i} k\cdot x} \big\>  \big\<\sigma_k \cdot \nabla \omega, D_\omega \rho_t \big\>.$$
A first thing is to know in which sense the above series is convergent. We shall prove

\begin{theorem}\label{thm-integrability}
Assume that the gradient estimate \eqref{thm-main-result-2.2} holds in the case $\gamma =2$. Then the series
  $$\<b(\omega), D_\omega \rho_t\>= 2\pi {\rm i} \sum_{k\in \Z^2_0} \big\<\omega, {\rm e}^{2\pi {\rm i} k\cdot x} \big\>  \big\<\sigma_k \cdot \nabla \omega, D_\omega \rho_t \big\>$$
converge in $L^2\big([0,T] \times H^{-1-}(\T^2), \d t \otimes \mu\big)$.
\end{theorem}

On the other hand, it seems impossible to establish a similar convergence result for $\gamma>2$. Therefore, a natural problem arises, namely

\begin{conjecture}
The gradient estimate \eqref{thm-main-result-2.2} holds when $\gamma =2$ in \eqref{vector-fields}.
\end{conjecture}

For the moment, we do not know how to solve this problem. For instance, the assumption \textbf{(H2)} is not satisfied in this case. Accordingly, the passage from the Stratonovich equation \eqref{weak-Euler-vorticity} to the It\^o equation produces an extra term which diverges at a logarithmic order. To summarize, with Theorem \ref{thm-integrability} we have identified a new example of transport type noise, namely (1.8) with $\gamma=2$, which is very promising for the purpose of regularization by noise, but it is at the border of the regularity class ($\gamma>2$) where Stratonovich noise has a meaning and where we can prove Theorems \ref{thm-main-result} and \ref{thm-main-result-2}. With these partial results we hope to promote research on this new type of noise.

The paper is organized as follows. In Section 2, we recall some facts about the stochastic dynamics of $N$-point vortices. More precisely, Section 2.1 is concerned with stochastic point vortices with an initial distribution which converging weakly to the white noise measure $\mu$, and Section 2.2 studies the case of general initial distributions, which is the basis for the approximation argument in later parts of the paper. We provide the proof of Theorem \ref{thm-main-result} in Section 3 which mainly follows the arguments in \cite[Section 4.2]{Flandoli}. The two assertions of Theorem \ref{thm-main-result-2} will be proved in Sections 4.1 and 4.2 respectively. In particular, the proof of assertion (ii) constitutes the main technical part of the current work, and it is done by first approximating the singular Biot--Savart kernel $K$ with smooth ones, and then letting the number $N$ of point vortices tend to infinity. Finally, we prove Theorem \ref{thm-integrability} in Section 5 by making use of the facts that, under the white noise measure $\mu$, the family $\big\{\big\<\omega, e^{2\pi {\rm i} k\cdot x} \big\>\big\}_{k\in \Z^2_0}$ consists of i.i.d. standard Gaussian r.v.'s and is an orthonormal basis of $L^2\big(H^{-1-}(\T^2), \mu\big)$.

\section{Stochastic point vortex dynamics}

According to \cite{MP}, in the singular case that the vorticity $\omega_0$ is given by $N\geq 2$ point vortices, the Euler equations \eqref{Euler-vorticity} can be interpreted as the finite dimensional dynamics in $(\T^2)^N$:
  $$\frac{\d X^{i,N}_t}{\d t} = \frac1{\sqrt N} \sum_{j=1}^N \xi_j K\big(X^{i,N}_t -X^{j,N}_t\big),\quad i=1,\ldots, N,$$
with initial condition $\big(X^{1,N}_0, \ldots, X^{N,N}_0\big)\in (\T^2)^N \setminus \Delta_N$, where $\xi = (\xi_1, \ldots, \xi_N) \in (\R \setminus \{0\})^N$ and
  $$\Delta_N= \big\{(x_1,\ldots, x_n)\in (\T^2)^N: \mbox{there are } i\neq j \mbox{ such that } x_i=x_j \big\}$$
is the generalized diagonal. The authors gave in \cite{MP} an example in the case $N=3$, which shows that the above system with different initial positions coincide in finite time. Nevertheless, it is well posed for $\big( {\rm Leb}_{\T^2}^{\otimes N } \big)$-a.e. starting point in $(\T^2)^N \setminus \Delta_N$.

For the stochastic Euler equations \eqref{stoch-Euler-vorticity}, the random version of the point vortex dynamics is given by
  \begin{equation}\label{stoch-point-vortex}
  \d X^{i,N}_t = \frac1{\sqrt N} \sum_{j=1}^N \xi_j K\big(X^{i,N}_t -X^{j,N}_t\big)\,\d t + \sum_{j=1}^N \sigma_j\big(X^{i,N}_t \big) \circ \d W^j_t,\quad i=1,\ldots, N.
  \end{equation}
Here, we use only a finite number of noises, because the stochastic equations with infinitely many noises may not admit a solution under the assumptions \textbf{(H1)} and \textbf{(H2)}. One can of course use a different number of noises, but the intuition is that this number should tend to $\infty$ as $N$ increases. A heuristic discussion of the relationship between \eqref{stoch-point-vortex} and
  \begin{equation}\label{stoch-Euler-vorticity-finite}
  \d \omega_t+ u_t\cdot \nabla \omega_t\, \d t +\sum_{j=1}^N \sigma_j\cdot \nabla\omega_t\circ \d W^j_t=0
  \end{equation}
can be found in \cite[Section 2.3]{FGP}. Roughly speaking, let $\big(X^{1,N}_t, \ldots, X^{N,N}_t\big)$ be the solution of \eqref{stoch-point-vortex} and set
  $$\omega^N_t= \frac1{\sqrt N} \sum_{i=1}^N \xi_i \delta_{X^{i,N}_t},$$
then for any $\phi\in C^\infty(\T^2)$, by applying the It\^o formula, $\omega^N_t$ satisfies
  \begin{equation}\label{weak-Euler-vorticity-finite}
  \d \<\omega_t,\phi\>= \<\omega_t, u_t\cdot \nabla \phi\>\, \d t +\sum_{j=1}^N \<\omega_t,\sigma_j\cdot \nabla\phi\> \circ \d W^j_t.
  \end{equation}

Fix any $N\in \N$ and denote by $\lambda_N= \Leb_{\T^2}^{\otimes N }$ which is a probability measure on $(\T^2)^N$.

\begin{theorem}\label{thm-invariance}
For every $(\xi_1,\ldots, \xi_N)\in (\R \setminus \{0\})^N$ and for $\lambda_N$-a.e. $\big(X^{1,N}_0, \ldots, X^{N,N}_0\big)\in (\T^2)^N \setminus \Delta_N$, almost surely, the system \eqref{stoch-point-vortex} has a unique strong solution $\big(X^{1,N}_t, \ldots, X^{N,N}_t\big)$ for all $t\geq 0$. Moreover, if the initial data $\big(X^{1,N}_0, \ldots, X^{N,N}_0\big)$ is a random variable distributed as $\lambda_N$, but is independent of the Brownian motions $\big\{(W^j_t)_{t\geq 0}: 1\leq j\leq N \big\}$, then $\big(X^{1,N}_t, \ldots, X^{N,N}_t\big)$ is a stationary process with invariant marginal law $\lambda_N$.
\end{theorem}

\begin{proof}
Note that our hypothesis \textbf{(H1)} is the same as the first one of \cite[Hypothesis 1]{FGP}, hence the first result follows from \cite[Theorem 8]{FGP}. We denote by $\varphi_t(X_0)$ the strong solution to \eqref{stoch-point-vortex} with initial condition $X_0\in (\T^2)^N \setminus \Delta_N$ when the solution exists. We remark that we do not need the ellipticity assumption in \cite[Hypothesis 1]{FGP}, since the existence of solution to \eqref{stoch-point-vortex} for a.e. starting point is enough for our purpose.

For proving the second assertion, let $K^\delta$ be the approximation of $K$ given in \cite[Section 3.2]{FGP} and $\varphi^\delta_t$ the flow of diffeomorphisms generated by \eqref{stoch-point-vortex} with $K$ replaced by $K^\delta$. Since the vector fields involved in \eqref{stoch-point-vortex} are divergence free, for any smooth function $h$ on $(\T^2)^N$, we have a.s. (cf. \cite[Lemma 3]{FGP})
  $$\int_{(\T^2)^N} h\big(\varphi^\delta_t(X_0)\big)\,\d X_0 = \int_{(\T^2)^N} h(Y)\,\d Y,\quad t\geq 0.$$
Therefore,
  $$\E\int_{(\T^2)^N} h\big(\varphi^\delta_t(X_0)\big)\,\d X_0 = \int_{(\T^2)^N} h(Y)\,\d Y,\quad t\geq 0.$$
For $(\lambda_N\otimes \P)$-a.s. $(X_0, \theta)\in (\T^2)^N \times \Theta$, we have $\varphi^\delta_t(X_0, \theta)\to \varphi_t(X_0, \theta)$ as $\delta\to 0$, see the proof of \cite[Theorem 8]{FGP}. Letting $\delta\to 0$ in the above equality leads to
  $$\int_{(\T^2)^N} h(Y)\,\d Y= \E\int_{(\T^2)^N} h\big(\varphi_t(X_0)\big)\,\d X_0 = \int_{(\T^2)^N} P^N_t h (X_0)\,\d X_0,$$
where $P^N_t$ is the semigroup associated to the system \eqref{stoch-point-vortex}. This implies that $\lambda_N$ is the invariant measure of $P^N_t$. The stationarity follows from the fact that the equations \eqref{stoch-point-vortex} are of time-homogeneous Markovian type.
\end{proof}

\subsection{Stochastic point vortices with initial distribution converging to white noise} \label{sect-white-noise}

On the probability space $(\Theta, \mathcal F, \P)$, let $\{\xi_n\}$ be an i.i.d. sequence of $N(0,1)$ r.v.'s and $\{X^n_0\}$ be an i.i.d. sequence of $\T^2$-valued r.v.'s, independent of $\{\xi_n\}$ and uniformly distributed. Both families are independent on the Brownian motions $\big\{(W^j_t)_{t\geq 0}: j\in \N \big\}$. For every $N\in \N$, denote by
  $$\lambda_N^0= \big(N(0,1) \otimes \Leb_{\T^2}\big)^{\otimes N }$$
the law of the random vector
  $$\big( (\xi_1, X^1_0), \ldots, (\xi_N, X^N_0)\big).$$
Let us consider the measure-valued vorticity field
  $$\omega^N_0 = \frac1{\sqrt N} \sum_{n=1}^N \xi_n \delta_{X^n_0}.$$
As mentioned in \cite[Remark 20]{Flandoli}, $\omega^N_0$ can be regarded as a r.v. taking values in the space $H^{-1-}(\T^2)$ whose law is denoted by $\mu_N^0$. Denote by $\mathcal M(\T^2)$ the space of signed measures on $\T^2$ with finite variation, and
  $$\mathcal M_N(\T^2)= \big\{\mu \in \mathcal M(\T^2)\,| \,\exists \, X\subset \T^2 \mbox{ such that } \#(X)=N \mbox{ and } {\rm supp}(\mu) = X\big\}.$$
We can define the map $\mathcal T_N: (\R \times \T^2)^N \to \mathcal M_N(\T^2) \subset H^{-1-}(\T^2)$ as
  \begin{equation}\label{mapping}
  \big( (\xi_1, X^1_0), \ldots, (\xi_N, X^N_0)\big) \mapsto \omega^N_0= \frac1{\sqrt N} \sum_{n=1}^N \xi_n \delta_{X^n_0},
  \end{equation}
then it holds that
  $$\mu_N^0 = (\mathcal T_N)_\# \lambda_N^0 = \lambda_N^0\circ \mathcal T_N^{-1}.$$
It is proved in \cite[Proposition 21]{Flandoli} that, for any $\delta>0$, as $N\to \infty$, $\omega_0^N$ converges in law on $H^{-1-\delta}(\T^2)$ to the white noise $\omega_{WN}$.

\begin{proposition}\label{prop-weak-convergence}
As $N\to \infty$, the probability measures $\mu_N^0$ converge weakly to $\mu$ on $H^{-1-}(\T^2)$.
\end{proposition}

\begin{proof}
\textbf{Step 1.} We first show that $\{\mu_N^0: N\in \N\}$ is tight on $H^{-1-}(\T^2)$. Fix an arbitrary $\varepsilon >0$. For every $n\in \N$, since $\mu_N^0$ converges weakly to $\mu$ on $H^{-1-1/n}(\T^2)$, it follows from \cite[p. 60, Theorem 5.2]{Billingsley} that the family $\{\mu_N^0: N\in \N\}$ is tight on $H^{-1-1/n}(\T^2)$. Therefore, there exists a compact set $K_{\varepsilon, n} \subset H^{-1-1/n}(\T^2)$ such that
  $$\sup_{N\in \N} \mu_N^0\big( H^{-1-1/n}(\T^2) \setminus K_{\varepsilon, n}\big) < \frac \varepsilon{2^n}.$$
Let $K_\varepsilon= \cap_{n\in \N} K_{\varepsilon, n}$; then $K_\varepsilon \subset \cap_{n\in \N} H^{-1-1/n}(\T^2) = H^{-1-}(\T^2)$. By the above inequality, we have for all $N\in \N$ that
  $$\mu_N^0\big( H^{-1-}(\T^2) \setminus K_\varepsilon\big) \leq \mu_N^0 \bigg(\bigcup_{n=1}^\infty \big( H^{-1-1/n}(\T^2) \setminus K_{\varepsilon, n}\big)\bigg) < \sum_{n=1}^\infty \frac \varepsilon{2^n} = \varepsilon.$$
Then the tightness of $\{\mu_N^0: N\in \N\}$ on $H^{-1-}(\T^2)$ will follow if we can show that $K_\varepsilon$ is compact in $H^{-1-}(\T^2)$. It is equivalent to show that $K_\varepsilon$ is sequentially compact in itself. Let $\{\omega_n: n\in \N\} \subset K_\varepsilon$ be an arbitrary sequence which will also be denoted by $\{\omega_{0,n}: n\in \N\}$.

Since $K_{\varepsilon, 1}$ is compact in $H^{-2}(\T^2)$ and $\{\omega_{0,n}: n\in \N\} \subset K_{\varepsilon, 1}$, we can find a subsequence $\{\omega_{1,n}: n\in \N\}$ of $\{\omega_{0,n}: n\in \N\}$, such that $\omega_{1,n}$ converges with respect to the norm $\|\cdot \|_{H^{-2}}$ to some $\omega_{1,0} \in K_{\varepsilon, 1}$.

Repeating this procedure inductively, for every $m\in \N$, we can find a subsequence $\{\omega_{m,n}: n\in \N\}$ of $\{\omega_{m-1, n}: n\in \N\}$ such that $\omega_{m,n}$ converges with respect to the norm $\|\cdot \|_{H^{-1-1/m}}$ to some $\omega_{m,0} \in K_{\varepsilon, m}$.

We claim that $\omega_{m, 0} = \omega_{m+1, 0}$ for all $m\in \N$. Indeed, on the one hand, since $\omega_{m+1,n}$ converge to $\omega_{m+1,0}$ with respect to the norm $\|\cdot \|_{H^{-1-1/(m+1)}}$, it also converge to $\omega_{m+1,0}$ with respect to the weaker norm $\|\cdot \|_{H^{-1-1/m}}$. On the other hand, as a subsequence of $\{\omega_{m,n}: n\in \N\}$, $\{\omega_{m+1,n}: n\in \N\}$ also converge in $H^{-1-1/m}(\T^2)$ to $\omega_{m,0}$. By the uniqueness of limit, we obtain $\omega_{m+1,0} = \omega_{m,0}$.

Therefore we can denote by $\omega_0$ the common limit of all the subsequences, which belongs to all $K_{\varepsilon, m}$, and hence is in $K_\varepsilon$. Now taking the diagonal subsequence $\{\omega_{n,n}: n\in \N\}$, we see that $\omega_{n,n}$ tends to $\omega_0$ with respect to all the norms $\|\cdot \|_{H^{-1-1/m}},\, m\geq 1$, hence the convergence holds in $H^{-1-}(\T^2)$ too. This shows that $K_\varepsilon$ is sequentially compact in itself.

\textbf{Step 2.} Let $\big\{\mu_{N_k}^0: k\in\N \big\}$ be a subsequence converging weakly to some $\nu$ on $H^{-1-}(\T^2)$. Then we have $\nu =\mu$. Indeed, for any bounded continuous function $F$ on $H^{-1-\delta}(\T^2)$, it is also continuous on $H^{-1-}(\T^2)$, hence $\lim_{k\to \infty} \int_{H^{-1-\delta}} F\,\d \mu_{N_k}^0 = \int_{H^{-1-\delta}} F\,\d \nu$. We conclude that $\mu_{N_k}^0$ converges weakly to $\nu$ on $H^{-1-\delta}(\T^2)$ for any $\delta >0$. This implies $\nu =\mu$. By the corollary of \cite[Theorem 5.1]{Billingsley}, the whole sequence $\{\mu_N^0: N\in\N\}$ converge weakly to $\mu$ on $H^{-1-}(\T^2)$.
\end{proof}

As a consequence of Theorem \ref{thm-invariance}, we can prove (cf. \cite[Proposition 22]{Flandoli} for the proof)

\begin{proposition}\label{prop-stationary}
Consider the stochastic point vortex dynamics \eqref{stoch-point-vortex} with random intensities $(\xi_1,\ldots, \xi_N)$ and random initial positions $\big( X^1_0, \ldots, X^N_0\big)$ distributed as $\lambda_N^0$. For a.s. value of $\big( (\xi_1, X^1_0), \ldots, (\xi_N, X^N_0)\big)$, the stochastic dynamics $\big( X^{1,N}_t, \ldots, X^{N,N}_t\big)$ is well defined in $\Delta_N^c$ for all $t\geq 0$, and the associated measure-valued vorticity $\omega^N_t$ satisfies the stochastic weak vorticity formulation of \eqref{weak-Euler-vorticity-finite}: for all $\phi \in C^\infty(\T^2)$,
  \begin{equation}\label{prop-stationary.1}
  \aligned
  \big\< \omega^N_t, \phi\big\> &= \big\< \omega^N_0, \phi\big\> +\int_0^t \big\<\omega^N_s\otimes \omega^N_s, H_\phi \big\>\, \d s +\sum_{j=1}^N \int_0^t \big\< \omega^N_s,\sigma_j \cdot \nabla \phi \big\>\,\d W^j_s\\
  &\hskip13pt + \frac12 \sum_{j=1}^N \int_0^t \big\<\omega^N_s,\sigma_j \cdot \nabla (\sigma_j \cdot \nabla \phi) \big\>\,\d s.
  \endaligned
  \end{equation}
The stochastic process $\omega^N_t$ is stationary in time, with the law $\mu_N^0$ at any time $t\geq 0$.
\end{proposition}

The following integrability properties of $\omega^N_t$ are proved in \cite[Lemma 23]{Flandoli} (except the second estimate, whose proof is similar to that of the first one).

\begin{lemma}\label{2-lem-integrability}
Assume $f:\T^2\times \T^2\to \R$ and $g:\T^2 \to \R$ are bounded and measurable, and $f$ is symmetric. Then, for every $p\geq 1$ and $\delta>0$, there are constants $C_p, C_{p,\delta}>0$ such that for all $t\in [0,T]$,
  $$\E\big[ \big| \big\<\omega^N_t \otimes \omega^N_t, f \big\> \big|^p \big] \leq C_p \|f\|_\infty^p,\quad \E\big[ \big| \big\<\omega^N_t , g \big\> \big|^p \big]\leq C_p \|g\|_\infty^p, \quad \E\big[ \big\|\omega^N_t \big\|_{H^{-1-\delta}}^p \big] \leq C_{p,\delta}.$$
Moreover,
  $$\E\big[ \big\<\omega^N_t \otimes \omega^N_t, f \big\>^2 \big]= \frac3N \! \int\! f^2(x,x)\,\d x +\frac{N-1}N \bigg[\int f(x,x)\,\d x \bigg]^2 + \frac{2(N-1)}N \! \int\!\int f^2(x,y)\,\d x\d y.$$
\end{lemma}

\subsection{Stochastic point vortices with general initial distribution} \label{sect-general}

In this part,  we shall consider stochastic point vortex dynamics \eqref{stoch-point-vortex} with more general initial distribution. Recall the definitions of $\lambda_N^0$, $\mu_N^0$ and $\mathcal T_N: (\R\times \T^2)^N \to \mathcal M_N(\T^2)$ in Section \ref{sect-white-noise}. The next lemma is taken from \cite[Lemma 29]{Flandoli}.

\begin{lemma}\label{lem-change-variable}
Let $\rho: H^{-1-}(\T^2) \to [0,\infty)$ be a measurable function with $\int_{H^{-1-}} \rho(\omega)\,\mu_N^0(\d \omega) <\infty$. Under the mapping $\mathcal T_N$, the measure $\lambda_N^\rho = (\rho\circ \mathcal T_N)\, \lambda_N^0$ has the image measure $\mu_N^\rho = \rho\, \mu_N^0$.
\end{lemma}

\begin{proof}
We denote by $(a,x)$ a typical element in $(\R\times \T^2)^N= \R^N \times (\T^2)^N$, where $a=(a_1,\ldots, a_N)\in \R^N$, $x=(x_1,\ldots, x_N)\in (\T^2)^N$. For every non-negative measurable function $F$, the change-of-variable formula yields
  \begin{align*}
  \int_{H^{-1-}(\T^2)} F(\omega)\, \mu_N^\rho(\d \omega) &= \int_{(\R\times \T^2)^N} F(\mathcal T_N(a, x)) \, \lambda_N^\rho (\d a, \d x)\\
  &= \int_{(\R\times \T^2)^N} F(\mathcal T_N(a,x)) \rho(\mathcal T_N(a,x))\, \lambda_N^0 (\d a,\d x)\\
  &= \int_{H^{-1-}(\T^2)} F(\omega) \rho(\omega)\, \mu_N^0(\d\omega). \qedhere
  \end{align*}
\end{proof}

Given $\rho_0\in C_b \big( H^{-1-}(\T^2) \big)$, $\rho_0 \geq 0$ and $\int \rho_0\,\d\mu =1$ ($\mu$ is the white noise Gaussian law on $H^{-1-}(\T^2)$), there is a normalizing constant $C_N>0$ such that $C_N \int \rho_0\,\d\mu_N^0 =1$. Since $\mu_N^0$ converges weakly to $\mu$ on $H^{-1-}(\T^2)$ by Proposition \ref{prop-weak-convergence}, we deduce that $\lim_{N\to \infty} C_N=1$. Let us consider the probability measure $C_N\, (\rho_0 \circ\mathcal T_N)\, \lambda_N^0$ on Borel sets of $(\R\times \T^2)^N$. By Lemma \ref{lem-change-variable} its image measure on $H^{-1-}(\T^2)$ under the map $\mathcal T_N$ is $C_N\, \rho_0\, \mu_N^0$. Recall that the stochastic point vortex dynamics \eqref{stoch-point-vortex} is well defined for $\lambda_N^0$-a.e. $\big( (\xi_1, X^1_0), \ldots, (\xi_N, X^N_0)\big) \in (\R\times \T^2)^N$. Hence it is well defined for a.e. $\big( (\xi_1, X^1_0), \ldots, (\xi_N, X^N_0)\big) \in (\R\times \T^2)^N$ with respect to $C_N \, (\rho_0\circ\mathcal T_N)\, \lambda_N^0$. Denote by $\omega^{N}_{\rho_0, t}$ the vorticity of this point vortex dynamics; the law of $\omega^{N}_{\rho_0, 0}$ on $\mathcal M_N(\T^2)\subset H^{-1-}(\T^2)$ is $C_N \, \rho_0\, \mu_N^0$.

\begin{lemma}\label{lem-law-point-vortices}
For any non-negative measurable function $F$ on $H^{-1-}(\T^2)$, one has
  $$\E\big[ F\big( \omega^{N}_{\rho_0, t} \big)\big] \leq C_N \|\rho_0\|_\infty \int_{\mathcal M_N(\T^2)} F(\omega)\, \mu_N^0 (\d\omega).$$
In particular, the law of $\omega^{N}_{\rho_0, t}$ on $\mathcal M_N(\T^2)$ has a density $\rho^N_t$ w.r.t. $\mu_N^0$.
\end{lemma}

\begin{proof}
For a given $\omega\in \mathcal M_N(\T^2)$, it corresponds to $N!$ different elements $(a,x)\in (\R\times \T^2)^N$. These elements differ from each other by a permutation. However, by changing accordingly the order of the equations in the system \eqref{stoch-point-vortex}, the solutions give rise to the same random measure-valued vorticity field at any time $t>0$. Thus, there exists a unique stochastic process $\Phi^N_t(\omega)$ associated to the system \eqref{stoch-point-vortex}, which is well defined for $\mu_N^0$-a.e. $\omega\in \mathcal M_N(\T^2)$. For any nonnegative measurable function $F: \mathcal M_N(\T^2) \to \R_+$, by the last assertion of Proposition \ref{prop-stationary},
  \begin{equation}\label{lem-law-point-vortices.1}
  \E \int_{\mathcal M_N(\T^2)} F\big(\Phi^N_t(\omega) \big) \, \mu_N^0(\d\omega)= \int_{\mathcal M_N(\T^2)} F(\omega)\, \mu_N^0(\d\omega).
  \end{equation}

Now, note that $\omega^{N}_{\rho_0, t}= \Phi^N_t\big(\omega^{N}_{\rho_0, 0} \big)$ where $\omega^{N}_{\rho_0, 0}$ is distributed as  $C_N \, \rho_0\, \mu_N^0$. Therefore,
  $$\aligned
  \E\big[ F\big( \omega^{N}_{\rho_0, t} \big)\big] &= \E\big[ F\big( \Phi^N_t \big(\omega^{N}_{\rho_0, 0} \big) \big)\big] = \int_{\mathcal M_N(\T^2)} \E\big[ F\big( \Phi^N_t (\omega) \big)\big] C_N \, \rho_0(\omega)\, \mu_N^0(\d\omega)\\
  &\leq C_N \|\rho_0\|_\infty \int_{\mathcal M_N(\T^2)} \E\big[ F\big( \Phi^N_t (\omega) \big)\big] \, \mu_N^0(\d\omega)\\
  &= C_N \|\rho_0\|_\infty \int_{\mathcal M_N(\T^2)} F(\omega)\, \mu_N^0(\d\omega),
  \endaligned$$
where the last equality follows from \eqref{lem-law-point-vortices.1}.
\end{proof}

We have the following useful estimates.

\begin{corollary}\label{cor-moment-estimate}
Assume $f:\T^2 \times \T^2\to \R$ and $g:\T^2 \to \R$ are bounded and measurable, and $f$ is symmetric. Then for any $p\geq 1$ and $\delta \in (0,1)$,  there exist $C_{\rho_0, p}, C_{\rho_0, p ,\delta}>0$ such that for all $t\in [0,T]$,
  $$\aligned
  \E\big[ \big|\big\< \omega^N_{\rho_0,t} \otimes \omega^N_{\rho_0,t}, f\big\>\big|^p \big] &\leq C_{\rho_0, p} \|f\|_\infty^p,\\
  \E\big[ \big|\big\< \omega^N_{\rho_0,t} , g\big\>\big|^p \big] &\leq C_{\rho_0, p} \|g\|_\infty^p,\\
  \E\big[ \big\|\omega^N_{\rho_0,t} \big\|_{H^{-1-\delta}}^p \big] &\leq C_{\rho_0, p ,\delta}.
  \endaligned$$

\end{corollary}

\begin{proof}
Since $\lim_{N\to \infty} C_N=1$, we have $C_0 =\sup_{N\geq 1} C_N <\infty$. Applying Lemma \ref{lem-law-point-vortices} with $F(\omega) = |\< \omega \otimes \omega, f \>|^p$, then  we deduce the first result from the estimate in Lemma \ref{2-lem-integrability} with $C_{\rho_0, p} = C_0 \|\rho_0\|_\infty C_{p}$. The last two estimates follow in the same way.
\end{proof}

\section{Proof of Theorem \ref{thm-main-result}}

For simplification of notations, we shall write in the sequel $\omega^N_t$ instead of $\omega^{N}_{\rho_0, t}$ given in Section \ref{sect-general}, since $\rho_0$ is fixed.

The difference of the proof, compared to that of \cite[Theorem 24]{Flandoli}, is that the process $\<\omega^N_t, \phi\>$ does not have differentiable trajectories, hence we shall use fractional Sobolev spaces and apply another compactness criterion proved in \cite[p. 90, Corollary 9]{Simon}. We state it here in our context.

Take $\delta\in (0,1)$ and $\kappa>5$ (this choice is due to estimates below) and consider the spaces
  $$X=H^{-1-\delta/2}(\T^2),\quad B=H^{-1-\delta}(\T^2),\quad Y=H^{-\kappa}(\T^2).$$
Then $X\subset B\subset Y$ with compact embeddings and we also have, for a suitable constant $C>0$ and for
  \begin{equation}\label{eq-theta}
  \theta= \frac{\delta/2}{\kappa -1-\delta/2},
  \end{equation}
the interpolation inequality
  $$\|\omega\|_B \leq C \|\omega\|_X^{1-\theta} \|\omega\|_Y^\theta,\quad \omega\in X.$$
These are the preliminary assumptions of \cite[p. 90, Corollary 9]{Simon}. We consider here a particular case:
  $$\mathcal S= L^{p_0}(0,T; X)\cap W^{1/3,4}(0,T; Y),$$
where for $0< \alpha <1$ and $p\geq 1$,
  $$W^{\alpha,p}(0,T; Y)=\bigg\{f: \, f\in L^p(0,T; Y) \mbox{ and } \int_0^T\! \int_0^T \frac{\|f(t)-f(s)\|_Y^p}{|t-s|^{\alpha p+1}}\,\d t\d s <\infty\bigg\}.$$

\begin{lemma}\label{lem-embedding}
Let $\delta\in (0,1)$ and $\kappa>5$ be given. If
  $$p_0> \frac{12(\kappa -1-3\delta/2)}\delta,$$
then $\mathcal S$ is compactly embedded into $C\big([0,T], H^{-1-\delta}(\T^2) \big)$.
\end{lemma}

\begin{proof}
Recall that $\theta$ is defined in \eqref{eq-theta}. In our case, we have $s_0=0, r_0=p_0$ and $s_1=1/3, r_1=4$. Hence $s_\theta = (1-\theta)s_0 +\theta s_1= \theta/3$ and
  $$\frac1{r_\theta} = \frac{1-\theta}{r_0} + \frac\theta{r_1} = \frac{1-\theta}{p_0} + \frac\theta 4.$$
It is clear that for $p_0$ given above, it holds $s_\theta> 1/r_\theta$, thus the desired result follows from the second assertion of \cite[Corollary 9]{Simon}.
\end{proof}

For $N\geq 1$, let $Q^N$ be the law of $\omega^N_\cdot$ on $\mathcal X:= C\big([0,T], H^{-1-}(\T^2) \big)$. We want to prove that the family $\big\{Q^N\big\}_{N\geq 1}$ is tight in $\mathcal X$.

\begin{lemma}\label{lem-tight}
The family $\big\{Q^N\big\}_{N\geq 1}$ is tight in $\mathcal X$ if and only if it is tight in $C\big([0,T], H^{-1-\delta}(\T^2) \big)$ for any $\delta>0$.
\end{lemma}

The proof is similar to Step 1 of the proof of Proposition \ref{prop-weak-convergence}. In view of the above two lemmas, it is sufficient to prove that  $\big\{Q^N\big\}_{N\geq 1}$ is bounded in probability in $W^{1/3,4}\big(0,T; H^{-\kappa}(\T^2) \big)$ and in each $L^{p_0}\big(0,T; H^{-1-\delta}(\T^2) \big)$ for any $p_0>0$ and $\delta>0$.

First we show that the family $\big\{Q^N\big\}_{N\geq 1}$ is bounded in probability in $L^{p_0}\big(0,T; H^{-1-\delta}(\T^2) \big)$. We have by Corollary \ref{cor-moment-estimate} that
  \begin{equation}\label{sec-3.1}
  \E\bigg[\int_0^T \big\|\omega^N_t \big\|_{H^{-1-\delta}}^{p_0} \,\d t\bigg] = \int_0^T \E\big[ \big\|\omega^N_t \big\|_{H^{-1-\delta}}^{p_0} \big]\,\d t \leq C_{\rho_0, p_0, \delta} T,\quad \mbox{for all } N\geq 1.
  \end{equation}
By Chebyshev's inequality, we obtain the boundedness in probability of the family $\big\{Q^N\big\}_{N\geq 1}$ in $L^{p_0}\big(0,T; H^{-1-\delta}(\T^2) \big)$.

Next, we prove the boundedness in probability in $W^{1/3,4}\big(0,T; H^{-\kappa}(\T^2) \big)$. Again by the Chebyshev inequality, it suffices to show that
  $$\sup_{N\geq 1} \E \bigg[\int_0^T \big\|\omega^N_t \big\|_{H^{-\kappa}}^4 \,\d t +\int_0^T\! \int_0^T \frac{\big\| \omega^N_t- \omega^N_s \big\|_{H^{-\kappa}}^4} {|t-s|^{7/3}}\,\d t\d s\bigg] <\infty. $$
In view of \eqref{sec-3.1}, we see that it is sufficient to establish a uniform estimate on the expectation $\E \big\|\omega^N_t- \omega^N_s \big\|_{H^{-\kappa}}^4$.

\begin{lemma}\label{lem-estimate}
Under the assumption {\rm \textbf{(H2)}}, for any $\phi\in C^\infty(\T^2)$, we have
  $$\E\big[ \big\<\omega^N_t- \omega^N_s, \phi \big\>^4 \big]\leq C (t-s)^2\big( \|\nabla \phi\|_\infty^4 + \|\nabla^2 \phi\|_\infty^4 \big).$$
\end{lemma}

\begin{proof}
The equation \eqref{prop-stationary.1} holds for $\big(\omega^N_t \big)_{0\leq t\leq T}$, thus
  \begin{equation}\label{lem-estimate-1}
  \aligned
  \big\<\omega^N_t- \omega^N_s, \phi\big\> &= \int_s^t \big\<\omega^N_r\otimes \omega^N_r, H_\phi\big\>\, \d r +\sum_{j=1}^N \int_s^t \big\<\omega^N_r,\sigma_j \cdot \nabla \phi \big\>\,\d W^j_r\\
  &\hskip13pt + \frac12 \sum_{j=1}^N \int_s^t \big\<\omega^N_r,\sigma_j \cdot \nabla (\sigma_j \cdot \nabla \phi) \big\>\,\d r.
  \endaligned
  \end{equation}
First, H\"older's inequality leads to
  \begin{equation}\label{lem-estimate-2}
  \aligned
  \E\bigg[\bigg(\int_s^t \big\<\omega^N_r\otimes \omega^N_r, H_\phi\big\>\, \d r\bigg)^{\! 4} \bigg]
  &\leq (t-s)^3 \E\bigg[\int_s^t \big\<\omega^N_r\otimes \omega^N_r, H_\phi \big\>^4\, \d r\bigg]\\
  &\leq (t-s)^3 \int_s^t C \|H_\phi\|_\infty^4 \,\d r \leq C (t-s)^4  \|\nabla^2 \phi\|_\infty^4,
  \endaligned
  \end{equation}
where the second inequality follows from Corollary \ref{cor-moment-estimate}. Next, by Burkholder's inequality,
  \begin{equation*}
  \aligned
  \E\bigg[\bigg(\sum_{j=1}^N \int_s^t \big\<\omega^N_r,\sigma_j \cdot \nabla \phi \big\>\,\d W^j_r\bigg)^{\! 4} \bigg]
  &\leq C\E \bigg[\bigg(\int_s^t \sum_{j=1}^N \big\<\omega^N_r,\sigma_j \cdot \nabla \phi\big\>^2\,\d r\bigg)^{\! 2} \bigg]\\
  &\leq C(t-s) \int_s^t \E \bigg[\bigg(\sum_{j=1}^N \big\<\omega^N_r,\sigma_j \cdot \nabla \phi \big\>^2 \bigg)^{\! 2} \bigg] \d r.
  \endaligned
  \end{equation*}
We have by Cauchy's inequality and Corollary \ref{cor-moment-estimate} that
  $$  \aligned
  \E \bigg[\bigg(\sum_{j=1}^N \big\<\omega^N_r,\sigma_j \cdot \nabla \phi \big\>^2 \bigg)^{\! 2} \bigg]
  &= \sum_{j,k=1}^N \E \big[ \big\<\omega^N_r,\sigma_j \cdot \nabla \phi \big\>^2 \big\<\omega^N_r,\sigma_k \cdot \nabla \phi \big\>^2 \big] \\
  &\leq \sum_{j,k=1}^N \big[\E \big\<\omega^N_r,\sigma_j \cdot \nabla \phi\big\>^4\big]^{1/2} \big[\E \big\<\omega^N_r,\sigma_k \cdot \nabla \phi\big\>^4\big]^{1/2}\\
  &\leq C\bigg(\sum_{j=1}^N \|\sigma_j \cdot \nabla \phi\|_\infty^2\bigg)^{\! 2} \leq \tilde C \|\nabla \phi\|_\infty^4,
  \endaligned $$
where the last inequality follows from \textbf{(H2)}. Substituting this estimate into the above inequality yields
  \begin{equation}\label{lem-estimate-3}
  \E\bigg[\bigg(\sum_{j=1}^N \int_s^t \big\<\omega^N_r,\sigma_j \cdot \nabla \phi \big\>\,\d W^j_r\bigg)^{\! 4} \bigg] \leq C(t-s)^2 \|\nabla \phi\|_\infty^4.
  \end{equation}

Finally, by H\"older's inequality,
  $$\aligned
  \E\bigg[\bigg(\sum_{j=1}^N \int_s^t \big\<\omega^N_r,\sigma_j \cdot \nabla (\sigma_j \cdot \nabla \phi) \big\> \,\d r\bigg)^{\! 4} \bigg]
  &\leq (t-s)^3 \int_s^t \E \bigg[\bigg(\sum_{j=1}^N \big\<\omega^N_r, \sigma_j \cdot \nabla (\sigma_j \cdot \nabla \phi) \big\>\bigg)^{\! 4} \bigg] \d r\\
  &\leq (t-s)^3 \int_s^t \bigg[ \sum_{j=1}^N \Big( \E\, \big<\omega^N_r, \sigma_j \cdot \nabla (\sigma_j \cdot \nabla \phi) \big>^4 \Big)^{\frac14} \bigg]^4 \d r.
  \endaligned$$
Since
  $$\aligned \Big( \E \big<\omega^N_r, \sigma_j \cdot \nabla (\sigma_j \cdot \nabla \phi) \big>^4 \Big)^{\frac14}
  &\leq C \big\|\sigma_j \cdot \nabla (\sigma_j \cdot \nabla \phi) \big\|_\infty \\
  &\leq C \Big(\|\sigma_j\|_\infty^2 \|\nabla^2 \phi\|_\infty + \|\sigma_j\cdot \nabla\sigma_j\|_\infty \|\nabla \phi\|_\infty \Big),
  \endaligned $$
we have by \textbf{(H2)} that
  $$\aligned
  &\hskip13pt \E\bigg[\bigg(\sum_{j=1}^N \int_s^t \big\<\omega^N_r,\sigma_j \cdot \nabla (\sigma_j \cdot \nabla \phi) \big\>\,\d r\bigg)^{\! 4} \bigg] \\
  &\leq C (t-s)^3 \int_s^t \bigg[\sum_{j=1}^N \Big(\|\sigma_j\|_\infty^2 \|\nabla^2 \phi\|_\infty + \|\sigma_j\cdot \nabla\sigma_j\|_\infty \|\nabla \phi\|_\infty \Big) \bigg]^4 \,\d r\\
  &\leq C (t-s)^4 \big(\|\nabla^2 \phi\|_\infty + \|\nabla \phi\|_\infty \big)^4.
  \endaligned$$
Combining this estimate together with \eqref{lem-estimate-1}--\eqref{lem-estimate-3}, we obtain the desired estimate.
\end{proof}

Applying Lemma \ref{lem-estimate} with $\phi(x)= e_k(x)= {\rm e}^{2\pi {\rm i} k\cdot x}$ leads to
  $$\E\big[ \big| \big\<\omega^N_t- \omega^N_s, e_k \big\> \big|^4 \big] \leq C (t-s)^2 |k|^8, \quad k\in \Z^2_0 = \Z^2 \setminus \{0\}.$$
As a result, by Cauchy's inequality,
  $$\aligned
  \E \big( \big\|\omega^N_t- \omega^N_s \big\|_{H^{-\kappa}}^4 \big) &= \E\bigg[\bigg( \sum_k \big(1+|k|^2 \big)^{-\kappa} \big|\big\<\omega^N_t- \omega^N_s, e_k \big\> \big|^2 \bigg)^{\! 2} \bigg]\\
  &\leq \bigg(\sum_k \big(1+|k|^2 \big)^{-\kappa}\bigg) \sum_k \big(1+|k|^2 \big)^{-\kappa} \E \big[ \big|\big\<\omega^N_t- \omega^N_s, e_k \big\> \big|^4 \big]\\
  &\leq \tilde C (t-s)^2\sum_k \big(1+|k|^2 \big)^{-\kappa} |k|^8 \leq \hat C (t-s)^2,
  \endaligned$$
since $2\kappa -8 >2$ due to the choice of $\kappa$. Consequently,
  $$\E \bigg[\int_0^T\! \int_0^T \frac{ \big\|\omega^N_t- \omega^N_s \big\|_{H^{-\kappa}}^4} {|t-s|^{7/3}}\,\d t\d s\bigg] \leq \hat C \int_0^T\! \int_0^T \frac{|t-s|^2} {|t-s|^{7/3}}\,\d t\d s <\infty.$$
The proof of the boundedness in probability of $\big\{Q^N\big\}_{N\geq 1}$ in $W^{1/3,4}\big(0,T; H^{-\kappa}(\T^2) \big)$ is complete.

We have shown that the family $\big\{ Q^N\big\}_{N\in \N}$ is bounded in probability in $L^{p_0}\big(0,T; H^{-1-\delta/2} \big)\cap W^{1/3,4} \big(0,T; H^{-\kappa}\big)$ for any $p_0>0$ and $\delta>0$, hence it is tight in $C\big([0,T], H^{-1-\delta} \big)$ for any $\delta>0$. Lemma \ref{lem-tight} implies that $\big\{ Q^N\big\}_{N\in \N}$ is tight in $\mathcal X= C\big([0,T], H^{-1-} \big)$.

Since we are dealing with the SDEs \eqref{stoch-point-vortex}, we need to consider $Q^N$ together with the distribution of Brownian motions. Although we use only finitely many Brownian motions in \eqref{stoch-point-vortex}, here we consider for simplicity the whole family $\big\{ (W^j_t)_{0\leq t\leq T}: j\in \N \big\}$. To this end, we assume $\R^\infty$ is endowed with the metric
  $$d_\infty(a,b)= \sum_{n=1}^\infty \frac{|a_n-b_n| \wedge 1}{2^n}, \quad a,b \in \R^\infty.$$
Then $(\R^\infty, d_\infty(a,b))$ is separable and complete (see \cite[p. 9, Example 1.2]{Billingsley}). The distance in $\mathcal Y:= C\big([0,T], \R^\infty \big)$ is given by
  $$d_{\mathcal Y}(w,\hat w) = \sup_{t\in [0,T]} d_\infty(w(t), \hat w(t)),\quad w, \hat w \in \mathcal Y,$$
which makes $\mathcal Y$ a Polish space. Denote by $\mathcal W$ the law on $\mathcal Y$ of the sequence of independent Brownian motions $\big\{ (W^j_t)_{0\leq t\leq T}: j\in \N \big\}$.

To simplify the notations, we write $W_\cdot= (W_t)_{0\leq t\leq T}$ for the whole sequence of processes $\big\{ (W^j_t)_{0\leq t\leq T}: j\in \N \big\}$ in $\mathcal Y$. Denote by $P^N$ the joint law of $\big(\omega^N_\cdot, W_\cdot \big)$ on $\mathcal X \times \mathcal Y,\, N\geq 1$. Since the marginal laws $\big\{ Q^N \big\}_{N\in \N}$ and $\{\mathcal W\}$ are respectively tight on $\mathcal X$ and $\mathcal Y$, we conclude that $\big\{ P^N \big\}_{N\in \N}$ is tight on $\mathcal X \times \mathcal Y$. By Skorokhod's representation theorem, there exists a subsequence $\{N_k \}_{k\in \N}$ of integers, a probability space $\big(\hat \Theta, \hat{\mathcal F}, \hat \P \big)$ and stochastic processes $\big(\hat \omega^{N_k}_\cdot, \hat W^{N_k}_\cdot\big)$ on this space with the corresponding laws $P^{N_k}$, and converging $\hat\P$-a.s. in $\mathcal X\times \mathcal Y$ to a limit $\big(\hat\omega_\cdot, \hat W_\cdot \big)$. We are going to prove that $\big(\hat\omega_\cdot, \hat W_\cdot \big)$, or more precisely another closely defined process, is the solution claimed by Theorem \ref{thm-main-result}.

As in \cite{Flandoli}, we need to enlarge the probability space $\big(\hat \Theta, \hat{\mathcal F}, \hat \P \big)$ so that it contains certain independent r.v.'s we need. Denote by $\big(\tilde \Theta, \tilde {\mathcal F}, \tilde \P \big)$ a probability space on which, for every $N\geq 1$, it is defined a uniformly distributed random permutation $\tilde s_N: \tilde \Theta \to \Sigma_N$, where $\Sigma_N$ is the permutation group of order $N$. Define the product probability space
  \begin{equation}\label{product-space}
  (\Theta, \mathcal F, \P ) =\big(\hat \Theta\times \tilde \Theta, \hat{\mathcal F}\otimes \tilde {\mathcal F}, \hat \P \otimes \tilde \P\big)
  \end{equation}
and the new processes
  $$\big(\omega^{N_k}, W^{N_k}\big) = \big(\hat \omega^{N_k}, \hat W^{N_k}\big)\circ \pi_1, \quad (\omega, W) = \big(\hat \omega, \hat W \big)\circ \pi_1, \quad  s_N = \tilde s_N \circ \pi_2,$$
where $\pi_1$ and $\pi_2$ are the projections on $\hat \Theta\times \tilde \Theta$. Here, we slightly abuse the notations by denoting the final probability spaces and processes like the original ones. We shall clarify in the sequel which ones we are investigating.

First, we have the following simple result.

\begin{lemma}\label{lem-absolute-continuity}
For every $t\in [0,T]$, the law $\mu_t$ of $\omega_t$ on $H^{-1-}(\T^2)$ is absolutely continuous with respect to the law $\mu$ of white noise, with a bounded density denoted by $\rho_t$.
\end{lemma}

\begin{proof}
Note that $\omega_\cdot$ is defined on the product probability space \eqref{product-space} but it has the same law with $\hat \omega_\cdot$. Hence it suffices to prove the assertion for $\hat \omega_t,\, t\in [0,T]$.

For every non-negative $F\in C_b\big(H^{-1-}(\T^2) \big)$, since $\hat \omega^{N_k}_t$ converges to $\hat \omega_t$ a.s., one has
  \begin{equation}\label{lem-absolute-continuity.1}
  \int F(\omega)\,\d\mu_t(\omega) = \hat\E \big[ F(\hat \omega_t)\big]= \lim_{k\to \infty} \hat\E \big[F\big( \hat \omega^{N_k}_t\big)\big] = \lim_{k\to \infty} \E \big[F\big( \omega^{N_k}_t\big)\big]
  \end{equation}
where $\hat\E$ is the expectation on $\big(\hat \Theta, \hat{\mathcal F}, \hat \P \big)$ and $\E$ the one on the original probability space. By Lemma \ref{lem-law-point-vortices} and Proposition \ref{prop-weak-convergence},
  \[\aligned
  \int F(\omega)\, \mu_t(\d\omega) & \leq \lim_{k\to \infty} C_{N_k} \|\rho_0\|_\infty \int F(\omega)\, \mu_{N_k}^0(\d\omega)
  = \|\rho_0\|_\infty \int F(\omega)\, \mu(\d\omega).
  \endaligned\]
This implies that $\mu_t\ll \mu$ with a density bounded by $\|\rho_0\|_\infty$.
\end{proof}

The following result identifies the structure of $\omega^{N_k}_\cdot$ as a sum of Dirac masses.

\begin{lemma}\label{lem-1}
The process $\omega^{N_k}_t$  on the new probability space can be represented in the form $\frac1{\sqrt {N_k}} \sum_{i=1}^{N_k} \xi_i \delta_{X^{i, N_k}_t}$, where
  $$\big(\big(\xi_1, X^{1,N_k}_0\big), \ldots, \big(\xi_{N_k}, X^{N_k,N_k}_0\big)\big)$$
is a random vector with law $\lambda_{N_k}^0$ and $\big(X^{1,N_k}_t, \ldots, X^{N_k,N_k}_t\big)$ solves the stochastic system \eqref{stoch-point-vortex} with the initial condition $\big(X^{1,N_k}_0, \ldots, X^{N_k,N_k}_0 \big)$ and new Brownian motions $\big\{\big( W^{N_k,j}_t \big): 1\leq j\leq N_k \big\}$ defined above.
\end{lemma}

\begin{proof}
Repeating \textbf{Step 1} of the proof of \cite[Lemma 28]{Flandoli}, we can find a family of  random elements $\big( \hat\xi_1, \hat X^{1, N_k}_\cdot \big), \ldots, \big( \hat\xi_{N_k}, \hat X^{N_k, N_k}_\cdot \big)$ in $\R\times C\big([0,T], \T^2\big)$, such that $\hat \omega^{N_k}_t = \frac1{\sqrt {N_k}} \sum_{i=1}^{N_k} \hat\xi_i \delta_{\hat X^{i, N_k}_t}$. The first claim will be proved after a redefinition of the random elements.

Next we follow the arguments of Krylov \cite[Section 2.6, p. 89]{Krylov}. Consider the filtration defined on the original probability space $(\Theta, \mathcal F, \P)$:
  $$\mathcal F_t= \sigma\big( (\xi_n, X^n_0): n\in \N\big) \vee \sigma \big( W_s: s\leq t\big),\quad t\in [0,T],$$
where $(\xi_n, X^n_0),\, n\in \N$ are given at the beginning of Section \ref{sect-white-noise}. Recall that we denote by $W_t$ the sequence of Brownian motions $\big\{W^j_t: j\in \N\big\}$. The processes $\big(\omega^N_t, W_t \big)$ are adapted to the filtration $(\mathcal F_t)_{0\leq t\leq T}$. Fix any $t_0\in [0,T)$. The increments of $W_s$ after the time $t_0$ is independent on $\mathcal F_{t_0}$. Therefore, the processes $\big(\omega^N_t, W_t \big)\, (t\leq t_0)$ do not depend on the increments of $W_s$ after the time $t_0$. Due to the coincidence of finite dimensional distributions, the processes $\big(\hat \omega^{N_k}_t, \hat W^{N_k}_t \big)\, (t\leq t_0)$  do not depend on the increments of $ \hat W^{N_k}_s$ after the time $t_0$. This property  holds in the limit process, i.e. for the process $\big(\hat\omega_\cdot, \hat W_\cdot \big)$. For the sake of convenience, we also denote $\big(\hat\omega_\cdot, \hat W_\cdot \big)$ by $\big(\hat \omega^{N_0}_t, \hat W^{N_0}_t \big)$. The above arguments imply that, for all $k\geq 0$ and any $j\in \N$, $\hat W^{N_k,j}_t$ is a Brownian motion with respect to the filtration $\hat{\mathcal F}^{N_k}_t$, which is the completion of $\sigma\big(\hat \omega^{N_k}_s, \hat W^{N_k}_s: s\leq t\big),\, t\in [0,T]$. Moreover, for all $k\geq 0$ and $s\leq t$, $\hat \omega^{N_k}_s$ is $\hat{\mathcal F}^{N_k}_t$-measurable. Since $\hat \omega^{N_k}_s$ is continuous with respect to $s$, it is a progressively measurable process with respect to $\hat{\mathcal F}^{N_k}_t$. Therefore, the stochastic integrals involved below make sense.

Since the original process $\omega^{N_k}_t$ satisfies \eqref{prop-stationary.1}, which implies
  $$\aligned \E\bigg[\sup_{t\in [0,T]} \bigg| & \big\<\omega^{N_k}_t, \phi\big\> - \big\<\omega^{N_k}_0, \phi\big\> - \int_0^t \! \int_{\T^2} \! \int_{\T^2} \nabla\phi(x) \cdot K(x-y)\, \omega^{N_k}_s(\d x) \omega^{N_k}_s(\d y) \d s\\
  &- \sum_{j=1}^{N_k} \int_0^t \big\<\omega^{N_k}_s, \sigma_j \cdot \nabla \phi\big\>\,\d W^j_s - \frac12 \sum_{j=1}^{N_k} \int_0^t \big\<\omega^{N_k}_s, \sigma_j \cdot \nabla (\sigma_j\cdot \nabla \phi)\big\>\,\d s\bigg| \wedge 1\bigg]=0
  \endaligned$$
for all $\phi\in C^\infty(\T^2)$, the same property holds for the new processes $\big(\hat \omega^{N_k}_t, \hat W^{N_k}_t\big)$, because they have the same finite dimensional distributions with $\big(\omega^{N_k}_t, W_t\big)$. Hence, $\hat \P$-a.s., it holds
  \begin{equation}\label{lem-1.1}
  \aligned \sup_{t\in [0,T]} \bigg| & \big\< \hat\omega^{N_k}_t, \phi\big\> - \big\< \hat\omega^{N_k}_0, \phi\big\> - \int_0^t \! \int_{\T^2} \! \int_{\T^2} \nabla\phi(x) \cdot K(x-y)\, \hat\omega^{N_k}_s(\d x) \hat\omega^{N_k}_s(\d y) \d s\\
  &- \sum_{j=1}^{N_k} \int_0^t \big\< \hat\omega^{N_k}_s, \sigma_j \cdot \nabla \phi\big\>\,\d \hat W^{N_k,j}_s - \frac12 \sum_{j=1}^{N_k} \int_0^t \big\< \hat\omega^{N_k}_s, \sigma_j \cdot \nabla (\sigma_j\cdot \nabla \phi)\big\>\,\d s\bigg| =0
  \endaligned
  \end{equation}
on a dense countable set of $\phi\in C^\infty(\T^2)$. Using the structure $\hat \omega^{N_k}_t = \frac1{\sqrt{N_k}} \sum_{i=1}^{N_k} \hat\xi_i \delta_{\hat X^{i, N_k}_t}$, we conclude that $\big(\hat X^{1,N_k}_t, \ldots, \hat X^{N_k,N_k}_t\big)$ solves the stochastic system \eqref{stoch-point-vortex} with the Brownian motions $\big(\hat W^{N_k,j}_t\big)_{t\geq 0},\, 1\leq j\leq N_k$.

At this stage, we can get the final assertion by applying the so-called shuffling procedure, which amounts to redefining the r.v.'s and processes on the product probability space \eqref{product-space} by composition with random permutations given before Lemma \ref{lem-absolute-continuity}. The remaining part of the proof is the same as that of \cite[Lemma 28]{Flandoli}, thus we omit it here.
\end{proof}

Finally, we show that the processes $(\omega, W)$ defined on the new probability space \eqref{product-space} is the $\rho$-white noise solution to the stochastic Euler equation.

\begin{proposition}\label{prop-1}
For any $\phi\in C^\infty(\T^2)$ and $t\in [0,T]$,
  $$\aligned \E\bigg[ \bigg|\<\omega_t, \phi\> - \<\omega_0, \phi\> & - \int_0^t \big\<\omega_s\otimes \omega_s, H_\phi \big\>\, \d s - \sum_{j=1}^\infty \int_0^t \big\<\omega_s, \sigma_j \cdot \nabla \phi\big\>\,\d W^j_s\\
  & - \frac12 \sum_{j=1}^\infty \int_0^t \big\<\omega_s, \sigma_j \cdot \nabla (\sigma_j\cdot \nabla \phi)\big\>\,\d s\bigg| \wedge 1\bigg]=0.
  \endaligned$$
\end{proposition}

This implies that \eqref{def-solution-1} holds a.s. at time $t$. Since the processes are continuous, we see that the identity holds uniformly in time, with probability one on the product space \eqref{product-space}. This will prove the assertion of Theorem \ref{thm-main-result}.

\begin{proof}[Proof of Proposition \ref{prop-1}]
We denote by $I$ the expectation on the left hand side of the identity. Recall the definition of $\big(\omega^{N_k}, W^{N_k}\big)$ before Lemma \ref{lem-absolute-continuity}. This process has the same distribution as that of $\big(\hat\omega^{N_k}, \hat W^{N_k}\big)$. Thus it follows from \eqref{lem-1.1} that for every $k\in \N$, it holds $\P$-a.s.,
  \begin{equation*}
  \aligned \big\< \omega^{N_k}_t, \phi\big\> - \big\< \omega^{N_k}_0, \phi\big\> & - \int_0^t \big\<\omega^{N_k}_s \otimes \omega^{N_k}_s, H_\phi \big\>\, \d s - \sum_{j=1}^{N_k} \int_0^t \big\< \omega^{N_k}_s, \sigma_j \cdot \nabla \phi\big\>\,\d  W^{N_k,j}_s\\
  &- \frac12 \sum_{j=1}^{N_k} \int_0^t \big\< \omega^{N_k}_s, \sigma_j \cdot \nabla (\sigma_j\cdot \nabla \phi)\big\>\,\d s=0.
  \endaligned
  \end{equation*}
Consequently, using the simple inequality $|a+b|\wedge 1 \leq |a|\wedge 1+ |b|\wedge 1 $ leads to
  \begin{eqnarray*}
  I &\leq& \E\big[\big|\<\omega_t, \phi\> - \big\< \omega^{N_k}_t, \phi\big\>\big|\wedge 1\big] + \E\big[\big|\<\omega_0, \phi\> - \big\< \omega^{N_k}_0, \phi\big\>\big|\wedge 1\big] \\
  && + \E \bigg[\bigg| \int_0^t \big\<\omega_s\otimes \omega_s, H_\phi \big\>\, \d s - \int_0^t \big\<\omega^{N_k}_s \otimes \omega^{N_k}_s, H_\phi \big\>\, \d s \bigg|\wedge 1 \bigg] \\
  && + \E \bigg[\bigg| \sum_{j=1}^\infty \int_0^t \big\<\omega_s, \sigma_j \cdot \nabla \phi\big\>\,\d W^j_s - \sum_{j=1}^{N_k} \int_0^t \big\< \omega^{N_k}_s, \sigma_j \cdot \nabla \phi\big\>\,\d  W^{N_k,j}_s \bigg|\wedge 1 \bigg] \\
  && + \E \bigg[\bigg| \frac12 \sum_{j=1}^\infty \int_0^t \big\<\omega_s, \sigma_j \cdot \nabla (\sigma_j\cdot \nabla \phi)\big\>\,\d s - \frac12 \sum_{j=1}^{N_k}\int_0^t \big\< \omega^{N_k}_s, \sigma_j \cdot \nabla (\sigma_j\cdot \nabla \phi)\big\>\,\d s \bigg|\wedge 1 \bigg].
  \end{eqnarray*}
We denote by $I^{N_k}_i,\, i=1,\ldots, 5$ the terms on the right hand side of the above inequality.

First, by the a.s. convergence of $\omega^{N_k}_\cdot$ to $\omega_\cdot$ in $C\big([0,T], H^{-1-}(\T^2)\big)$ we immediately get
  $$\lim_{k\to \infty} I^{N_k}_1 = \lim_{k\to \infty} I^{N_k}_2=0.$$
Next, to show that $I^{N_k}_3$ tends to 0, we consider a smooth approximation $H^\delta_\phi$ of $H_\phi$ (see \cite[Remark 9]{Flandoli}), with $H^\delta_\phi(x,x)=0$ for all $x\in \T^2$ and $\delta >0$. The a.s. convergence of $\omega^{N_k}_\cdot$ to $\omega_\cdot$ in $C\big([0,T], H^{-1-}(\T^2)\big)$ implies that of $\omega^{N_k}_\cdot \otimes \omega^{N_k}_\cdot$ to $\omega_\cdot \otimes \omega_\cdot$ in $C\big([0,T], H^{-2-}(\T^2\times \T^2)\big)$. Hence, for all $\delta>0$,
  \begin{equation*}
  \lim_{k\to \infty} \E\bigg[\bigg|\int_0^t \big\<\omega_s\otimes \omega_s, H^\delta_\phi \big\>\, \d s - \int_0^t \big\<\omega^{N_k}_s \otimes \omega^{N_k}_s, H^\delta_\phi \big\>\, \d s \bigg|\wedge 1 \bigg] =0.
  \end{equation*}
As a result,
  \begin{equation}\label{prop-1.1}
  \aligned
  \lim_{k\to \infty} I^{N_k}_3 &\leq \E \bigg[\bigg|\int_0^t \big\<\omega_s\otimes \omega_s, H^\delta_\phi -H_\phi \big\>\, \d s \bigg|\wedge 1 \bigg]\\
  &\hskip13 pt + \limsup_{k\to \infty} \E \bigg[\bigg|\int_0^t \big\<\omega^{N_k}_s \otimes \omega^{N_k}_s, H^\delta_\phi -H_\phi \big\>\, \d s \bigg|\wedge 1 \bigg]\\
  &=: I_{3,1} + I_{3,2}.
  \endaligned
  \end{equation}
We have
  $$I_{3,1} \leq \int_0^t \E\big| \big\<\omega_s\otimes \omega_s, H^\delta_\phi -H_\phi \big\> \big|\, \d s \leq \int_0^t \Big[ \E  \big\<\omega_s\otimes \omega_s, H^\delta_\phi -H_\phi \big\>^2 \Big]^{1/2}\, \d s. $$
Thus by Lemma \ref{lem-absolute-continuity} and \cite[Theorem 8]{Flandoli},
  \begin{equation}\label{prop-1.2}
  \aligned I_{3,1} &\leq \int_0^t \Big[ \E_\mu \big(\rho_s(\omega) \big\<\omega\otimes \omega, H^\delta_\phi -H_\phi \big\>^2 \big) \Big]^{1/2}\, \d s\\
  &\leq t \|\rho_0\|_\infty^{1/2} \Big[ \E_\mu \big\<\omega\otimes \omega, H^\delta_\phi -H_\phi \big\>^2 \Big]^{1/2} \to 0 \quad \mbox{as } \delta \to 0.
  \endaligned
  \end{equation}
Here $\E_\mu$ is the expectation on $H^{-1-}$ w.r.t. the white noise measure $\mu$. Similarly, using Lemma \ref{lem-law-point-vortices} we can show that
  $$\aligned
  I_{3,2}& \leq \limsup_{k\to \infty} \int_0^t \Big[ \E  \big\<\omega^{N_k}_s\otimes \omega^{N_k}_s, H^\delta_\phi -H_\phi \big\>^2 \Big]^{1/2}\, \d s\\
  &\leq \limsup_{k\to \infty}  \int_0^t \bigg[ C_{N_k} \|\rho_0\|_\infty \int_{H^{-1-}(\T^2)} \big\<\omega\otimes \omega, H^\delta_\phi -H_\phi \big\>^2\mu_{N_k}^0(\d\omega) \bigg]^{1/2}\, \d s.
  \endaligned$$
Now by the last assertion of Lemma \ref{2-lem-integrability} and the convention that $H_\phi(x,x)= H^\delta_\phi(x,x)=0$,
  $$I_{3,2} \leq t \sqrt{2\|\rho_0\|_\infty }\, \bigg[ \int_{\T^2} \int_{\T^2} \big(H^\delta_\phi -H_\phi \big)^2(x,y)\, \d x\d y \bigg]^{1/2} \to 0$$
as $\delta\to 0$. Combining this result with \eqref{prop-1.1} and \eqref{prop-1.2}, we obtain
  $$\lim_{k\to \infty} I^{N_k}_3 =0.$$

We turn to the simpler term $I^{N_k}_5$. Fix some big integer $J$. For $N_k >J$, we have
  \begin{equation}\label{prop-1.8}
  \aligned
  2 I^{N_k}_5 &\leq \E \bigg[\bigg| \sum_{j=1}^J \int_0^t \Big( \big\<\omega_s, \sigma_j \cdot \nabla (\sigma_j\cdot \nabla \phi) \big\> - \big\< \omega^{N_k}_s, \sigma_j \cdot \nabla (\sigma_j\cdot \nabla \phi)\big\> \Big) \d s \bigg|\wedge 1 \bigg] \\
  & \hskip13pt + \E \bigg[\bigg| \sum_{j=J+1}^\infty \int_0^t \big\<\omega_s, \sigma_j \cdot \nabla (\sigma_j\cdot \nabla \phi)\big\> \, \d s \bigg|\wedge 1 \bigg] \\
  &\hskip13pt + \E \bigg[\bigg| \sum_{j=J+1}^{N_k} \int_0^t \big\<\omega^{N_k}_s, \sigma_j \cdot \nabla (\sigma_j\cdot \nabla \phi)\big\> \, \d s \bigg|\wedge 1 \bigg] \\
  &=: I^{N_k}_{5,1} +I^{N_k}_{5,2} +I^{N_k}_{5,3}.
  \endaligned
  \end{equation}
Analogous to $I^{N_k}_1$ and $I^{N_k}_2$, since $\sum_{j=1}^J \sigma_j \cdot \nabla (\sigma_j\cdot \nabla \phi)$ is smooth on $\T^2$, we have
  \begin{equation}\label{prop-1.5.1}
  \lim_{k\to \infty} I^{N_k}_{5,1} =0.
  \end{equation}
Next, by Lemma \ref{lem-absolute-continuity},
  $$I^{N_k}_{5,2} \leq \sum_{j=J+1}^\infty \int_0^t \E \big| \big\<\omega_s, \sigma_j \cdot \nabla (\sigma_j\cdot \nabla \phi)\big\> \big| \, \d s \leq \|\rho_0\|_\infty \sum_{j=J+1}^\infty \int_0^t \E_\mu \big| \big\<\omega, \sigma_j \cdot \nabla (\sigma_j\cdot \nabla \phi)\big\> \big| \, \d s. $$
By Cauchy's inequality and the definition of the white noise measure $\mu$, we have
  $$\aligned
  \E_\mu \big| \big\<\omega, \sigma_j \cdot \nabla (\sigma_j\cdot \nabla \phi)\big\> \big| &\leq \Big(\E_\mu \big| \big\<\omega, \sigma_j \cdot \nabla (\sigma_j\cdot \nabla \phi)\big\> \big|^2 \Big)^{1/2} \\
  & = \bigg(\int_{\T^2} \big|\sigma_j \cdot \nabla (\sigma_j\cdot \nabla \phi) \big|^2 \,\d x \bigg)^{1/2} \\
  &\leq \|\sigma_j\|_\infty^2 \|\nabla^2\phi\|_{L^2(\T^2)} + \|\sigma_j\cdot \nabla\sigma_j\|_\infty \|\nabla \phi\|_{L^2(\T^2)}.
  \endaligned $$
Therefore, for any $k$,
  \begin{equation}\label{prop-1.5.2}
  I^{N_k}_{5,2} \leq C_\phi T \|\rho_0\|_\infty \sum_{j=J+1}^\infty \big(\|\sigma_j\|_\infty^2 + \|\sigma_j\cdot \nabla\sigma_j \|_\infty \big).
  \end{equation}
In the same way, using Lemma \ref{lem-law-point-vortices}, we can prove that, for all $N_k>J$,
  $$I^{N_k}_{5,2} \leq C'_\phi T \|\rho_0\|_\infty \sum_{j=J+1}^\infty \big(\|\sigma_j\|_\infty^2 + \|\sigma_j\cdot \nabla\sigma_j \|_\infty \big).$$
Combining this estimate with \textbf{(H2)} and \eqref{prop-1.8}--\eqref{prop-1.5.2}, first letting $k\to \infty$ in \eqref{prop-1.8}, and then $J\to \infty$, we obtain
  $$\lim_{k\to \infty} I^{N_k}_5 =0.$$

It remains to deal with the more difficult term $I^{N_k}_4$. Fix again $J\in \N$. We have, for all $N_k >J$,
  \begin{equation}\label{prop-1.4}
  \aligned
  I^{N_k}_4 & \leq \E \bigg[\bigg| \sum_{j=J+1}^\infty \int_0^t \big\<\omega_s, \sigma_j \cdot \nabla \phi\big\>\,\d W^j_s \bigg| \bigg] + \E \bigg[\bigg| \sum_{j=J+1}^{N_k} \int_0^t \big\< \omega^{N_k}_s, \sigma_j \cdot \nabla \phi\big\>\,\d  W^{N_k,j}_s \bigg| \bigg]\\
  &\hskip13pt + \E\bigg[\bigg| \sum_{j=1}^J \int_0^t \big\<\omega_s, \sigma_j \cdot \nabla \phi\big\>\,\d W^j_s - \sum_{j=1}^J \int_0^t \big\< \omega^{N_k}_s, \sigma_j \cdot \nabla \phi\big\>\,\d  W^{N_k,j}_s \bigg| \bigg]\\
  &=: I^{N_k}_{4,1} +I^{N_k}_{4,2} +I^{N_k}_{4,3}.
  \endaligned
  \end{equation}
By the Cauchy inequality and It\^o isometry,
  $$\aligned
  I^{N_k}_{4,1}&\leq \bigg\{\E \bigg[\bigg| \sum_{j=J+1}^\infty \int_0^t \big\<\omega_s, \sigma_j \cdot \nabla \phi\big\>\,\d W^j_s \bigg|^2 \bigg] \bigg\}^{1/2} = \bigg\{\int_0^t \sum_{j=J+1}^\infty \E \big\<\omega_s, \sigma_j \cdot \nabla \phi\big\>^2 \,\d s  \bigg\}^{1/2}.
  \endaligned$$
Lemma \ref{lem-absolute-continuity}  implies that
  \begin{equation}\label{prop-1.5}
  \aligned
  I^{N_k}_{4,1}&\leq \bigg\{ \int_0^t \sum_{j=J+1}^\infty \|\rho_0\|_\infty \E_\mu \big\<\omega, \sigma_j \cdot \nabla \phi\big\>^2 \, \d s  \bigg\}^{1/2}\\
  &= \sqrt{ t \|\rho_0\|_\infty}\,  \bigg(\sum_{j=J+1}^\infty \int_{\T^2} |\sigma_j \cdot \nabla \phi|^2 \, \d x\bigg)^{1/2}\\
  & \leq \sqrt{ t \|\rho_0\|_\infty}\, \|\nabla \phi\|_{L^2(\T^2)} \bigg(\sum_{j=J+1}^\infty \|\sigma_j \|_\infty^2 \bigg)^{1/2}.
  \endaligned
  \end{equation}
Similarly, by Corollary \ref{cor-moment-estimate},
  \begin{equation}\label{prop-1.6}
  \aligned
  I^{N_k}_{4,2} &\leq  \bigg\{\int_0^t \sum_{j=J+1}^{N_k} \E \big\<\omega^{N_k}_s, \sigma_j \cdot \nabla \phi\big\>^2 \,\d s  \bigg\}^{1/2} \\
  &\leq \bigg\{\int_0^t \sum_{j=J+1}^\infty C\|\rho_0\|_\infty \| \sigma_j \cdot \nabla \phi\|_\infty^2 \,\d s  \bigg\}^{1/2}\\
  &\leq \sqrt{Ct \|\rho_0\|_\infty} \,\|\nabla\phi\|_\infty  \bigg(\sum_{j=J+1}^\infty \|\sigma_j \|_\infty^2 \bigg)^{1/2}.
  \endaligned
  \end{equation}
Finally, we consider the quantity $I^{N_k}_{4,3}$. Denote by $\eta_s = \big(\big\<\omega_s, \sigma_1 \cdot \nabla \phi\big\>, \ldots, \big\<\omega_s, \sigma_J \cdot \nabla \phi\big\>\big)$; then
  $$\aligned
  \E\big(|\eta_s|^4 \big) & = \E \bigg[ \bigg( \sum_{j=1}^J \big\<\omega_s, \sigma_j \cdot \nabla \phi\big\>^2\bigg)^2\bigg] = \sum_{j,l=1}^J \E \big(\big\<\omega_s, \sigma_j \cdot \nabla \phi\big\>^2 \big\<\omega_s, \sigma_l \cdot \nabla \phi\big\>^2\big) \\
  &\leq \sum_{j,l=1}^J \big( \E \big\<\omega_s, \sigma_j \cdot \nabla \phi\big\>^4\big)^{1/2} \big( \E \big\<\omega_s, \sigma_l \cdot \nabla \phi\big\>^4\big)^{1/2} = \bigg\{ \sum_{j=1}^J \big( \E \big\<\omega_s, \sigma_j \cdot \nabla \phi\big\>^4 \big)^{1/2} \bigg\}^2.
  \endaligned$$
Again by Lemma \ref{lem-absolute-continuity},
  $$\E\big(|\eta_s|^4 \big) \leq \bigg\{ \sum_{j=1}^J \big( \|\rho_0\|_\infty C \|\sigma_j \cdot \nabla \phi\|_\infty^4 \big)^{1/2}\bigg\}^2 \leq C\|\rho_0\|_\infty \|\nabla\phi\|_\infty^4 \bigg\{ \sum_{j=1}^J  \|\sigma_j \|_\infty^2 \bigg\}^2. $$
As a result, $\int_0^T \E \big(|\eta_s|^4\big) \,\d s<\infty$. Similarly, setting $\eta^k_s = \big(\big\<\omega^{N_k}_s, \sigma_1 \cdot \nabla \phi\big\>, \ldots, \big\<\omega^{N_k}_s, \sigma_J \cdot \nabla \phi\big\>\big)$ and using Corollary \ref{cor-moment-estimate}, we can show that
  $$\sup_{k\in \N} \int_0^T \E \big(\big |\eta^k_s\big|^4 \big) \,\d s<\infty.$$
Since $\big(\omega^{N_k}_\cdot, W^{N_k}_\cdot\big)$ converge to $(\omega_\cdot, W_\cdot)$ a.s., we can apply \cite[Lemma 3.2]{Luo} to get
  $$\lim_{k\to \infty} \E\bigg[\bigg| \sum_{j=1}^J \int_0^t \big\<\omega_s, \sigma_j \cdot \nabla \phi\big\>\,\d W^j_s - \sum_{j=1}^J \int_0^t \big\< \omega^{N_k}_s, \sigma_j \cdot \nabla \phi\big\>\,\d  W^{N_k,j}_s \bigg|^2 \bigg] =0.$$
Therefore, first letting $k\to \infty$ and then $J\to \infty$ in \eqref{prop-1.4}, we deduce from the above limit and \eqref{prop-1.5}, \eqref{prop-1.6} that
  $$\lim_{k\to \infty} I^{N_k}_4 =0.$$
We have shown that all the terms $I^{N_k}_i,\, i=1, \ldots, 5$ tend to 0 as $k\to \infty$. The proof is complete.
\end{proof}

\section{Proof of Theorem \ref{thm-main-result-2}}

We prove the two assertions of Theorem \ref{thm-main-result-2} in the following two subsections respectively.

\subsection{Proof of assertion (i)}

Let $\omega_\cdot$ be a solution of the stochastic Euler equations \eqref{stoch-Euler-vorticity} given by Theorem \ref{thm-main-result}, with the associated density $\rho_\cdot$. Let $F\in \mathcal{FC}_{P,T}$ be of the form $F(t,\omega) =\sum_{i=1}^m g_i(t) f_i(\<\omega, \phi_1\>, \ldots , \<\omega, \phi_n\>)$. For every $j=1, \ldots, n$, we have
  $$\d\<\omega_t, \phi_j\> = \<\omega_t\otimes \omega_t, H_{\phi_j}\>\, \d t +\sum_{k=1}^\infty \<\omega_t,\sigma_k \cdot \nabla \phi_j\>\,\d W^k_t +\frac12 \sum_{k=1}^\infty \big\<\omega_t,\sigma_k \cdot \nabla (\sigma_k \cdot \nabla \phi_j) \big\>\,\d t.$$
To simplify the notations, we denote by $\Phi = (\phi_1, \ldots , \phi_n)$ and $\<\omega, \Phi\> = (\<\omega, \phi_1\>, \ldots , \<\omega, \phi_n\>)$. Then, the It\^o formula leads to
  $$\aligned
  \d f_i(\<\omega_t, \Phi\>) &= \sum_{j=1}^n \partial_j f_i (\<\omega_t, \Phi\>) \bigg[\<\omega_t\otimes \omega_t, H_{\phi_j}\>\, \d t +\sum_{k=1}^\infty \<\omega_t,\sigma_k \cdot \nabla \phi_j\>\,\d W^k_t\\
  &\hskip 50pt +\frac12 \sum_{k=1}^\infty \big\<\omega_t,\sigma_k \cdot \nabla (\sigma_k \cdot \nabla \phi_j) \big\>\,\d t\bigg]\\
  &\hskip13pt + \frac12 \sum_{j,l=1}^n \partial_{j,l} f_i (\<\omega_t, \Phi\>) \sum_{k=1}^\infty \<\omega_t,\sigma_k \cdot \nabla \phi_j\> \<\omega_t,\sigma_k \cdot \nabla \phi_l\> \,\d t.
  \endaligned$$
By the definition of $\<D_\omega F(t,\omega), b(\omega)\>$,
  \begin{equation}\label{proof.0}
  \aligned
  \d F(t,\omega_t)&= \sum_{i=1}^m g'_i(t)f_i(\<\omega_t, \Phi\>)\,\d t + \sum_{i=1}^m g_i(t)\,\d f_i(\<\omega_t, \Phi\>)\\
  &= \big[(\partial_t F)(t, \omega_t) + \<b(\omega_t), D_\omega F(t,\omega_t)\> \big]\,\d t + \d M(t)\\
  &\hskip13pt + \frac12 \sum_{k=1}^\infty \sum_{i=1}^m g_i(t) \bigg[\sum_{j=1}^n \partial_j f_i (\<\omega_t, \Phi\>) \big\<\omega_t,\sigma_k \cdot \nabla (\sigma_k \cdot \nabla \phi_j) \big\> \\
  &\hskip50pt + \sum_{j,l=1}^n \partial_{j,l} f_i (\<\omega_t, \Phi\>) \<\omega_t,\sigma_k \cdot \nabla \phi_j\> \<\omega_t,\sigma_k \cdot \nabla \phi_l\> \bigg] \d t,
  \endaligned
  \end{equation}
where the martingale part
  \begin{equation}\label{martingale}
  \d M(t) = \sum_{i=1}^m g_i(t) \sum_{j=1}^n \partial_j f_i (\<\omega_t, \Phi\>) \sum_{k=1}^\infty \<\omega_t,\sigma_k \cdot \nabla \phi_j\>\,\d W^k_t.
  \end{equation}

\begin{lemma}\label{sec-4.1-lem}
Assume that $G\in \mathcal{FC}_{P}$ has the form $G(\omega)= g(\<\omega, \Phi\>) = g(\<\omega, \phi_1\>, \ldots , \<\omega, \phi_n\>)$. Then
  $$\aligned
  \big\< \sigma_k \cdot \nabla \omega, D_\omega \<\sigma_k \cdot \nabla \omega, D_\omega G\> \big\>
  &= \sum_{j=1}^n \partial_j g(\<\omega, \Phi\>) \big\<\omega,\sigma_k \cdot \nabla (\sigma_k \cdot \nabla \phi_j) \big\> \\
  &\hskip13pt + \sum_{j,l=1}^n \partial_{j,l} g(\<\omega, \Phi\>) \<\omega,\sigma_k \cdot \nabla \phi_j\> \<\omega,\sigma_k \cdot \nabla \phi_l\>.
  \endaligned $$
\end{lemma}

\begin{proof}
Since $D_\omega G =\sum_{j=1}^n \partial_j g(\<\omega, \Phi\>) \phi_j$, we have
  $$\<\sigma_k \cdot \nabla \omega, D_\omega G\> = \sum_{j=1}^n \partial_j g(\<\omega, \Phi\>) \<\sigma_k \cdot \nabla \omega, \phi_j\> = - \sum_{j=1}^n \partial_j g(\<\omega, \Phi\>) \<\omega, \sigma_k \cdot \nabla \phi_j\>,$$
where the last equality is due to $\div(\sigma_k)=0$. Therefore,
  $$D_\omega \<\sigma_k \cdot \nabla \omega, D_\omega G\> = -\sum_{j=1}^n \<\omega, \sigma_k \cdot \nabla \phi_j\> \sum_{l=1}^n \partial_{l,j} g(\<\omega, \Phi\>) \phi_l - \sum_{j=1}^n \partial_j g(\<\omega, \Phi\>)\, (\sigma_k \cdot \nabla \phi_j).$$
As a result,
  $$\aligned
  \big\< \sigma_k \cdot \nabla \omega, D_\omega \<\sigma_k \cdot \nabla \omega, D_\omega G\>_{L^2} \big\>_{L^2}
  &= - \sum_{j,l=1}^n \partial_{l,j} g(\<\omega, \Phi\>) \<\omega, \sigma_k \cdot \nabla \phi_j\> \<\sigma_k \cdot \nabla \omega, \phi_l\> \\
  &\hskip13pt - \sum_{j=1}^n \partial_j g(\<\omega, \Phi\>) \<\sigma_k \cdot \nabla \omega, \sigma_k \cdot \nabla \phi_j\>.
  \endaligned$$
This immediately leads to the desired result by integration by parts.
\end{proof}

Using the above lemma, we obtain
  \begin{equation}\label{proof.1}
  \aligned
  \d F(t,\omega_t)&= \big[(\partial_t F)(t, \omega_t) + \<b(\omega_t), D_\omega F(t,\omega_t)\> \big]\,\d t + \d M(t)\\
  &\hskip13pt + \frac12 \sum_{k=1}^\infty \big\< \sigma_k \cdot \nabla \omega_t, D_\omega \<\sigma_k \cdot \nabla \omega_t, D_\omega F(t,\omega_t )\> \big\> \,\d t .
  \endaligned
  \end{equation}
Following the arguments in Remark \ref{1-rem-1} we can show that $M(t)$ is a square integrable martingale. Indeed, by the expression \eqref{martingale} of $M(t)$, it is sufficient to show that for each $i\in \{1,\ldots, m\}$ and $j\in \{1,\ldots, n\}$, one has
  $$I:=\sum_{k=1}^\infty \E \int_0^T \big|g_i(t) \partial_j f_i (\<\omega_t, \Phi\>) \<\omega_t,\sigma_k \cdot \nabla \phi_j\> \big|^2\,\d t <\infty.$$
Since the law of $\omega_t$ is $\rho_t \mu$ and $\|\rho_t\|_\infty \leq \|\rho_0\|_\infty$ for all $t\in [0,T]$, we have
  $$ I \leq \|g_i\|_\infty^2\|\rho_0\|_\infty T \sum_{k=1}^\infty \E_\mu \Big( \big|\partial_j f_i (\<\omega, \Phi\>) \<\omega,\sigma_k \cdot \nabla \phi_j\> \big|^2 \Big), $$
where $\E_\mu$ is the expectation on $H^{-1-}$ w.r.t. $\mu$. By Cauchy's inequality,
  $$\aligned
  I &\leq C \sum_{k=1}^\infty \Big( \E_\mu \big|\partial_j f_i (\<\omega, \Phi\>)\big|^4 \Big)^{1/2} \Big( \E_\mu \big| \<\omega, \sigma_k \cdot \nabla \phi_j\> \big|^4 \Big)^{1/2}\\
  &\leq CC_1 \sum_{k=1}^\infty \big(C_2 \|\sigma_k \cdot \nabla \phi_j\|_\infty^4 \big)^{1/2} \leq C' \| \nabla \phi_j\|_\infty^2 \sum_{k=1}^\infty \|\sigma_k\|_\infty^2 <\infty,
  \endaligned$$
where the second inequality is due to the facts that the function $\partial_j f_i$ has polynomial growth and $\<\omega, \Phi\>$ is a Gaussian random vector.

Integrating \eqref{proof.1} from $0$ and $T$ and taking expectation, we deduce from $F(T,\cdot)=0$ that
  $$\aligned
  0&= \E F(0,\omega_0) + \int_0^T \E\big[(\partial_t F)(t, \omega_t) + \<b(\omega_t), D_\omega F(t,\omega_t)\> \big]\,\d t\\
  &\hskip13pt + \frac12 \sum_{k=1}^\infty \int_0^T \E \big\< \sigma_k \cdot \nabla \omega_t, D_\omega \<\sigma_k \cdot \nabla \omega_t, D_\omega F(t,\omega_t )\> \big\> \,\d t \\
  &= \int F(0,\omega) \rho_0(\omega)\mu(\d\omega) + \int_0^T \!\int \big[(\partial_t F)(t, \omega) + \<b(\omega), D_\omega F(t,\omega)\> \big] \rho_t(\omega) \mu(\d\omega) \d t \\
  &\hskip13pt + \frac12 \sum_{k=1}^\infty \int_0^T \!\int \big\< \sigma_k \cdot \nabla \omega, D_\omega \<\sigma_k \cdot \nabla \omega, D_\omega F(t,\omega)\> \big\> \rho_t(\omega) \mu(\d\omega) \d t.
  \endaligned$$
The proof of assertion (i) is complete.

\begin{remark}\label{sec-4.1-rem}
We remark that each integral on the r.h.s. of the above equation is finite. For instance, since $\rho_t$ is bounded on $H^{-1-}$, uniformly in $t\in [0,T]$, by the assertion above Theorem \ref{thm-main-result-2}, we see that
  $$\int_0^T \!\int \big| \<b(\omega), D_\omega F(t,\omega)\> \rho_t(\omega) \big|\, \mu(\d\omega) \d t <\infty.$$
Next, to prove the finiteness of the last integral, by \eqref{proof.0} and \eqref{proof.1}, it is enough to show that for all $i\in \{1,\ldots, m\}$ and $j,l\in \{1,\ldots, n\}$,
  $$J_1:= \sum_{k=1}^\infty \int_0^T \! \int \big| \partial_j f_i (\<\omega, \Phi\>) \big\<\omega,\sigma_k \cdot \nabla (\sigma_k \cdot \nabla \phi_j) \big\> \,\rho_t(\omega) \big| \,\mu(\d\omega)\d t<\infty$$
and
  $$J_2:= \sum_{k=1}^\infty \int_0^T \! \int \big| \partial_{j,l} f_i (\<\omega, \Phi\>) \<\omega,\sigma_k \cdot \nabla \phi_j\> \<\omega,\sigma_k \cdot \nabla \phi_l\> \,\rho_t(\omega) \big| \,\mu(\d\omega)\d t<\infty.  $$
Here we only prove the second estimate. We have
  $$\aligned
  J_2 &\leq \|\rho_0\|_\infty\, T \sum_{k=1}^\infty \E_\mu \big| \partial_{j,l} f_i (\<\omega, \Phi\>) \<\omega,\sigma_k \cdot \nabla \phi_j\> \<\omega,\sigma_k \cdot \nabla \phi_l\> \big|  \\
  &\leq C\sum_{k=1}^\infty \Big(\E_\mu \big| \partial_{j,l} f_i (\<\omega, \Phi\>) \big|^2 \Big)^{1/2} \Big(\E_\mu \big| \<\omega, \sigma_k \cdot \nabla \phi_j\> \big|^4 \Big)^{1/4} \Big(\E_\mu \big| \<\omega, \sigma_k \cdot \nabla \phi_l\> \big|^4 \Big)^{1/4} \\
  & \leq C C_1 \sum_{k=1}^\infty C_2 \| \sigma_k \cdot \nabla \phi_j\|_\infty \| \sigma_k \cdot \nabla \phi_l\|_\infty <\infty,
  \endaligned$$
where the third inequality we use the fact that $\partial_{j,l} f_i $ has polynomial growth, and the last one is due to {\rm \textbf{(H2)}}.
\end{remark}

\subsection{Proof of assertion (ii)}

Our strategy is to prove the assertion in three steps:
\begin{itemize}
\item[(1)] Fix $N\in \N$. Prove the gradient estimate on $(\T^2)^N$ in the smooth case, i.e. the kernel $K$ in \eqref{stoch-point-vortex} is replaced by some smooth one $K^\delta$.
\item[(2)] Let $\delta\to 0$ to get the gradient estimate in the case of the singular Biot--Savart kernel $K$, and rewrite it in terms of the density $\rho^N_t$ of point vortices.
\item[(3)] Let $N\to \infty$ to obtain the desired result.
\end{itemize}

\textbf{Step 1: Smooth kernel $K^\delta$.} We fix $N\geq 1$ and let $K^\delta$ be the smooth kernel given in \cite[Section 3.2]{FGP}. For the moment we fix a family of vortex intensities $\xi = (\xi_1,\ldots, \xi_N)$. Consider \eqref{stoch-point-vortex} with $K$ replaced by $K^\delta$ and denote the solution flow by $X^\delta_t =\big(X^{\delta,1}_t, \ldots, X^{\delta,N}_t \big)$. It is well known that $X^\delta_t$ is a stochastic flow of diffeomorphisms on $(\T^2)^N$.

Define the vector fields $A^{(N)}_k$ on $(\T^2)^N$ as follows: for $x=(x_1, \ldots, x_N)\in (\T^2)^N$,
  $$A^{(N)}_k(x) = A^{(N)}_{\sigma_k} (x) = \big(\sigma_k(x_1), \ldots, \sigma_k(x_N)\big).$$
For simplicity we shall write $A_k,\, k\in\N$. We also define the drift vector field $A^\delta_0: (\T^2)^N \to (\R^2)^N$ by
  $$\big(A^\delta_0 \big)_i(x)= \frac1{\sqrt N} \sum_{j=1}^N \xi_j K^\delta(x_i-x_j),\quad x\in (\T^2)^N , \, 1\leq i\leq N.$$
Then the equation \eqref{stoch-point-vortex} can be simply written as
  $$\d X^\delta_t = A^\delta_0\big(X^\delta_t\big)\,\d t + \sum_{k=1}^N A_k\big(X^\delta_t\big) \circ \d W^k_t,\quad X^\delta_0= x\in (\T^2)^N.$$

\textbf{Smooth initial condition.} Let $v_0: (\T^2)^N \to \R$ be a smooth function. For $x\in (\T^2)^N$, define $v^\delta_t(x) = v_0 \big(X^{\delta, -1}_t (x)\big)$, where $X^{\delta, -1}_t$ is the inverse flow. We have (see \cite[pp. 103--106]{Bismut})
  $$\aligned\d v^\delta_t &= -\big(A^\delta_0\cdot \nabla v^\delta_t\big)\,\d t -\sum_{k=1}^N \big(A_k\cdot \nabla v^\delta_t \big)\circ \d W^k_t\\
  &= -\big(A^\delta_0\cdot \nabla v^\delta_t\big)\,\d t -\sum_{k=1}^N \big(A_k\cdot \nabla v^\delta_t\big)\, \d W^k_t +\frac12\sum_{k=1}^N \big[A_k\cdot \nabla \big(A_k\cdot \nabla v^\delta_t\big) \big] \,\d t.
  \endaligned$$
Here $\nabla= \nabla_{2N}$ is the gradient on $(\T^2)^N$. For $\phi\in C^{1,2}\big([0, T]\times (\T^2)^N\big)$, It\^o's formula leads to
  $$\d\big(\phi_t v^\delta_t \big)= \phi'_t\, v^\delta_t\,\d t -\phi_t\big(A^\delta_0\cdot \nabla v^\delta_t \big)\,\d t- \phi_t \sum_{k=1}^N \big(A_k\cdot \nabla v^\delta_t \big)\, \d W^k_t + \frac12\phi_t \sum_{k=1}^N \big[A_k\cdot \nabla \big(A_k\cdot \nabla v^\delta_t \big) \big] \,\d t.$$
Integrating from 0 to $T$ yields
  \begin{equation}\label{sec-4.1}
  \aligned
  \phi_T v^\delta_T &= \phi_0 v^\delta_0 + \int_0^T \phi'_t\, v^\delta_t\,\d t -\int_0^T \phi_t\,\big(A^\delta_0\cdot \nabla v^\delta_t \big)\,\d t - \sum_{k=1}^N \int_0^T \phi_t\, \big(A_k\cdot \nabla v^\delta_t \big)\, \d W^k_t\\
  &\hskip13pt + \frac12\sum_{k=1}^N \int_0^T\phi_t\, \big[A_k\cdot \nabla \big(A_k\cdot \nabla v^\delta_t \big) \big] \,\d t.
  \endaligned
  \end{equation}

Define
  $$u^\delta_t(x)= \E v^\delta_t(x)= \E\big[ v_0 \big(X^{\delta, -1}_t (x)\big) \big],\quad (t,x)\in [0,T]\times (\T^2)^N.$$
Denote by $\<\cdot, \cdot\>$ the inner product in $L^2\big((\T^2)^N, \lambda_N\big)$, where $\lambda_N = {\rm Leb}_{\T^2}^{\otimes N} $. By integrating \eqref{sec-4.1} on $(\T^2)^N$ and taking expectation we obtain
  \begin{equation}\label{eq-1}
  \aligned
  \big\<\phi_T, u^\delta_T \big\>&= \big\<\phi_0, u^\delta_0 \big\> + \int_0^T \big\<\phi'_t, u^\delta_t \big\>\,\d t -\int_0^T \big\<\phi_t, A^\delta_0\cdot \nabla u^\delta_t \big\>\,\d t\\
  &\hskip13pt + \frac12 \sum_{k=1}^N \int_0^T \big\<\phi_t, A_k\cdot \nabla \big(A_k\cdot \nabla u^\delta_t \big) \big\> \,\d t.
  \endaligned
\end{equation}
Note that $u^\delta_\cdot\in C^1\big([0,T], C^\infty\big((\T^2)^N \big) \big)$. Choosing $\phi_\cdot = u^\delta_\cdot$ leads to
  $$\aligned \big\|u^\delta_T \big\|_{L^2}^2&= \big\|u^\delta_0 \big\|_{L^2}^2 + \frac12 \int_0^T \frac\d{\d t} \big\|u^\delta_t \big\|_{L^2}^2\,\d t -\int_0^T \big\<u^\delta_t, A^\delta_0\cdot \nabla u^\delta_t \big\>\,\d t\\
  &\hskip13pt + \frac12 \sum_{k=1}^N \int_0^T \big\<u^\delta_t, A_k\cdot \nabla \big(A_k\cdot \nabla u^\delta_t \big) \big\> \,\d t.
  \endaligned$$
Since $\div_{2N}\big(A^\delta_0 \big)=0$, the third term on the r.h.s. vanishes. Applying the integration by parts in the last term yields
  \begin{equation}\label{gradient-estimate-1}
  \sum_{k=1}^N \int_0^T \big\|A_k\cdot \nabla u^\delta_t \big\|_{L^2}^2 \,\d t = \big\|u^\delta_0 \big\|_{L^2}^2 - \big\|u^\delta_T \big\|_{L^2}^2 \leq \|v_0\|_\infty^2.
  \end{equation}
Therefore, we obtain the gradient estimate in the case that $v_0\in C^\infty\big( (\T^2)^N \big)$.

\textbf{Continuous initial condition.} Assume now $v_0 \in C\big( (\T^2)^N \big)$. We take a sequence of smooth functions $v_n$ which converge uniformly to $v_0$, such that $\|v_n\|_\infty \leq \|v_0\|_\infty$. Then $u^{\delta,n}_t(x)= \E\big[ v_n\big(X^{\delta, -1}_t (x)\big) \big]$ satisfies \eqref{gradient-estimate-1}, i.e.
  \begin{equation}\label{eq-2}
  \sum_{k=1}^N \int_0^T \big\|A_k\cdot \nabla u^{\delta,n}_t \big\|_{L^2}^2 \,\d t \leq \|v_n\|_\infty^2 \leq \|v_0\|_\infty^2.
  \end{equation}
Define the set of integers $S_N = \{1,\ldots, N\}$. The above inequality shows that the sequence $\big\{\big(A_k\cdot \nabla u^{\delta,n}_t\big)(x)\, |\, (k,t,x)\in S_N \times [0,T] \times (\T^2)^N \big\}_{n\in \N}$ is bounded in $L^2\big(S_N \times [0,T] \times (\T^2)^N, \#\otimes \d t \otimes \lambda_N\big)$, where $\#$ is the counting measure on $S_N$ and $\lambda_N = {\rm Leb}_{\T^2}^{\otimes N }$. We denote this Hilbert space by $L^2\big(S_N \times [0,T] \times (\T^2)^N\big)$ for simplicity. Then, there exists a subsequence $\big\{\big(A_k\cdot \nabla u^{\delta,n_i}_t\big)(x) \big\}_{i\in \N}$ which converges weakly to some $\alpha^\delta \in L^2\big(S_N \times [0,T] \times (\T^2)^N\big)$, satisfying
  \begin{equation}\label{eq-2.5}
  \sum_{k=1}^N \int_0^T \big\|\alpha^\delta_k(t) \big\|_{L^2}^2 \,\d t \leq \|v_0\|_\infty^2.
  \end{equation}

Define the space of test functions by
  \begin{equation}\label{test-functions}
  \aligned
  \mathcal{C}_T(N)= \big\{\beta = (\beta_1, \ldots, \beta_N) \,|\, \beta_k\in C^{0,1}\big([0,T]\times (\T^2)^N \big) \mbox{ for all } 1\leq k\leq N \big\}.
  \endaligned
  \end{equation}
Then for any $\beta \in \mathcal{C}_T(N) $,
  $$\lim_{i\to \infty} \sum_{k=1}^N \int_0^T \!\! \int_{(\T^2)^N}  \big(A_k\cdot \nabla u^{\delta,n_i}_t\big)(x) \beta_k(t,x) \,\d x\d t = \sum_{k=1}^N \int_0^T \!\! \int_{(\T^2)^N} \alpha^\delta_k(t,x) \beta_k(t,x) \,\d x\d t.$$
Using the fact that $\div_{2N}(A_k) \equiv 0$ and integrating by parts give us
  $$\aligned
  \sum_{k=1}^N \int_0^T \!\! \int_{(\T^2)^N}  \big(A_k\cdot \nabla u^{\delta,n_i}_t\big)(x) \beta_k(t,x) \,\d x\d t
  &= - \sum_{k=1}^N \int_0^T \!\! \int_{(\T^2)^N}  u^{\delta,n_i}_t(x) (A_k\cdot \nabla \beta_k(t) )(x) \,\d x\d t \\
  &\to - \sum_{k=1}^N \int_0^T \!\! \int_{(\T^2)^N}  u^\delta_t(x) (A_k\cdot \nabla \beta_k(t) )(x) \,\d x\d t
  \endaligned$$
as $i\to \infty$, since $v_{n_i}$ converges uniformly to $v_0$. Here, $u^{\delta}_t(x)= \E\big[ v_0\big(X^{\delta, -1}_t (x)\big) \big]$. Summarizing the two limits above yields
  \begin{equation}\label{eq-3}
  \sum_{k=1}^N \int_0^T \!\! \int_{(\T^2)^N} \alpha^\delta_k(t,x) \beta_k(t,x) \,\d x\d t = - \sum_{k=1}^N \int_0^T \!\! \int_{(\T^2)^N}  u^\delta_t(x) (A_k\cdot \nabla \beta_k(t) )(x) \,\d x\d t,
  \end{equation}
which holds for any $\beta \in \mathcal{C}_T(N) $. This equality implies the weak limit $\alpha^\delta$ is independent on the choices of the sequence $\{v_n\}_{n\in \N}$ of smooth initial conditions and the subsequence $\{n_i\}_{i\in \N}$. Moreover, for any fixed $k\in S_N$, taking $\beta \in \mathcal{C}_T(N)$ such that $\beta_j \equiv 0$ for all $j\neq k$, we obtain
  $$\int_0^T \!\! \int_{(\T^2)^N} \alpha^\delta_k(t,x) \beta_k(t,x) \,\d x\d t = - \int_0^T \!\! \int_{(\T^2)^N}  u^\delta_t(x) (A_k\cdot \nabla \beta_k(t) )(x) \,\d x\d t.$$
Since the vector fields $\{A_k\}_{k\in S_N}$ are divergence free, we see that the following equalities
  \begin{equation}\label{eq-4}
  A_k\cdot \nabla u^\delta_\cdot  = \alpha^\delta_k, \quad k\in S_N
  \end{equation}
hold in the distributional sense. Combining this fact with \eqref{eq-2.5} yields the gradient estimate
  \begin{equation}\label{gradient-estimate-3}
  \sum_{k=1}^N \int_0^T \big\|A_k\cdot \nabla u^{\delta}_t \big\|_{L^2}^2 \,\d t \leq \|v_0\|_\infty^2.
  \end{equation}

\medskip

\textbf{Step 2: Non-smooth kernel $K$.} In this step we aim to extend the gradient estimate to the case where $K$ is the singular Biot--Savart kernel.  The proof is similar to the passage to the limit from smooth initial conditions to continuous ones.

Let $v_0\in C\big((\T^2)^N, \R_+ \big)$ be the initial probability density function of $X^\delta_0$. For any nonnegative continuous function $F$ on $(\T^2)^N$, we have
  \begin{equation}\label{density}
  \aligned
  \E \big[F\big( X^\delta_t\big)\big]& = \int_{(\T^2)^N} \E \big[F \big( X^\delta_t(x)\big)\big] v_0(x) \,\d x\\
  &= \E \int_{(\T^2)^N} F(y) v_0\big( X^{\delta,-1}_t(y)\big) \,\d y \leq \|v_0\|_\infty \int_{(\T^2)^N} F(y) \,\d y,
  \endaligned
  \end{equation}
where the second equality is due to the fact that $X^\delta_t$ preserves the volume measure of $(\T^2)^N$. By the proof of \cite[Theorem 8]{FGP}, for $\lambda_N$-a.e. $x\in (\T^2)^N$, we have a.s. $X^\delta_t(x)\to X_t(x)$ as $\delta \to 0$ for all $t\in [0,T]$. The dominated convergence theorem yields
  \begin{equation}\label{convergence}
  \aligned \lim_{\delta\to 0}\E \big[F\big( X^\delta_t\big)\big] &= \lim_{\delta\to 0}\int_{(\T^2)^N} \E \big[F \big( X^\delta_t(x)\big)\big] v_0(x) \,\d x\\
  &= \int_{(\T^2)^N} \E \big[F \big( X_t(x)\big)\big] v_0(x) \,\d x= \E F(X_t).\endaligned
  \end{equation}
Combining \eqref{density} and \eqref{convergence} we conclude that the law $\mu_t$ of $X_t$ is absolutely continuous w.r.t. $\lambda_N = {\rm Leb}_{\T^2}^{\otimes N }$, with a density function $u_t$ bounded by $\|v_0\|_\infty$. Moreover, the second equality in \eqref{density} shows that $u^\delta_t(x)= \E\big[ v_0\big( X^{\delta,-1}_t(x)\big) \big]$ is the density function of $X^\delta_t$. For general bounded continuous function $F$ on $(\T^2)^N$, analogous to \eqref{convergence}, we have
  \begin{equation}\label{eq-8}
  \lim_{\delta\to 0} \int_{(\T^2)^N} F(x) u^\delta_t(x)\,\d x= \int_{(\T^2)^N} F(x) u_t(x)\,\d x
  \end{equation}
which means that $ u^\delta_t$ converges weakly to $u_t$ as $\delta\to 0$.

Recall that $S_N= \{1,\ldots, N\}$. By \eqref{gradient-estimate-3}, the family $\big\{ \big(A_k\cdot \nabla u^{\delta}_t \big)(x)\, |\, (k,t,x) \in S_N \times [0,T] \times (\T^2)^N \big\}_{\delta>0}$ is bounded in the Hilbert space $L^2\big(S_N \times [0,T] \times (\T^2)^N \big)$. Thus, there exists a subsequence $\big\{ \big(A_k\cdot \nabla u^{\delta_i}_t \big)(x) \big\}_{i\in \N}$ which converges weakly to some function $\alpha \in L^2\big( S_N \times [0,T] \times (\T^2)^N \big)$, satisfying
  \begin{equation}\label{eq-9}
  \sum_{k=1}^N \int_0^T \|\alpha_k(t)\|_{L^2}^2 \,\d t \leq \|v_0\|_\infty^2.
  \end{equation}
Consequently, for any $\beta \in \mathcal{C}_T(N) $ (see \eqref{test-functions}),
  $$\lim_{i\to \infty} \sum_{k=1}^N \int_0^T \!\! \int_{(\T^2)^N} \big(A_k\cdot \nabla u^{\delta_i}_t \big)(x) \beta_k(t,x) \,\d x\d t = \sum_{k=1}^N \int_0^T \!\! \int_{(\T^2)^N} \alpha_k(t,x) \beta_k(t,x) \,\d x\d t.$$
By \eqref{eq-3} and \eqref{eq-4},
  $$\aligned
  \sum_{k=1}^N \int_0^T \!\! \int_{(\T^2)^N} \big(A_k\cdot \nabla u^{\delta_i}_t \big)(x) \beta_k(t,x) \,\d x\d t &= - \sum_{k=1}^N \int_0^T \!\! \int_{(\T^2)^N} u^{\delta_i}_t(x) (A_k\cdot \nabla \beta_k(t))(x) \,\d x\d t\\
  &\to - \sum_{k=1}^N \int_0^T \!\! \int_{(\T^2)^N} u_t(x) (A_k\cdot \nabla \beta_k(t))(x) \,\d x\d t
  \endaligned$$
as $i\to \infty$, where the last step follows from \eqref{eq-8}. Combining the two limits above yields that, for any $\beta \in \mathcal{C}_T(N) $,
  \begin{equation}\label{equation-1}
  \sum_{k=1}^N \int_0^T \!\! \int_{(\T^2)^N} \alpha_k(t,x) \beta_k(t,x) \,\d x\d t = - \sum_{k=1}^N \int_0^T \!\! \int_{(\T^2)^N} u_t(x) (A_k\cdot \nabla \beta_k(t))(x) \,\d x\d t.
  \end{equation}
As above, we deduce from this equality that $\alpha$ does not depend on the choice of the subsequence $\big\{ \big(A_k\cdot \nabla u^{\delta_i}_t \big)(x) \big\}_{i\in \N}$, and for all $k\in S_N$,
  \begin{equation}\label{equation-2}
  A_k\cdot \nabla u_\cdot = \alpha_k
  \end{equation}
holds in the distribution sense. Moreover, we deduce from \eqref{eq-9} the gradient estimate
  \begin{equation}\label{gradient-estimate-2.6}
  \sum_{k=1}^N \int_0^T \big\|A_k\cdot \nabla u_t \big\|_{L^2}^2 \,\d t \leq \|v_0\|_\infty^2.
  \end{equation}

\textbf{Random intensity vector $\xi=(\xi_1,\ldots, \xi_N)$.} In the above discussions we assume the intensity vector $\xi= (\xi_1, \ldots, \xi_N)$ is fixed. To be more precise, we shall write in the sequel $X^\xi_t(x) = \big(X^{\xi,1}_t(x), \ldots, X^{\xi,N}_t(x)\big)$ for the strong solution of \eqref{stoch-point-vortex} which is well defined for a.e. $x\in \Delta_N^c$, and $u^\xi_t$ the density of $X^\xi_t$ starting from the initial density $v_0 \in C \big((\T^2)^N\big)$.

Now we suppose $\xi$ is a random vector and the joint law of $(\xi, X_0)$ is
  $$\rho(a,x) \lambda_N^0(\d a, \d x) = \rho(a,x) p_N(a)\,\d a \d x,$$
where $\rho: (\R\times \T^2)^N \to \R_+$ is a bounded continuous probability density function w.r.t. $\lambda_N^0$, and $p_N(a)= (2\pi)^{-N/2} e^{-|a|^2/2}$. Then the marginal distribution of $\xi$ is
  $$\tilde p(a)= p_N(a) \int_{(\T^2)^N} \rho(a,x)\,\d x,$$
and the conditional distribution of $X_0$ given $\xi=a$ is
  \begin{equation}\label{conditional-density}
  v_a(x)= \frac{\rho(a,x) p_N(a)}{\tilde p(a)} = \frac{\rho(a,x)}{\int_{(\T^2)^N} \rho(a,x)\,\d x}.
  \end{equation}
Therefore, under the probability measure $\rho(a,x) \lambda_N^0(\d a, \d x)$ and given $\xi=a$, $v_a$ is the initial density of $X^a_0$. Let $u^a_t$ be the density of $X^a_t$ when the initial density of $X^a_0$ is given by $v_a(x)$ in \eqref{conditional-density}. According to \eqref{equation-1} and \eqref{equation-2}, for all $k\in \{1,\ldots, N\}$ and $\beta \in C^{0,1}\big([0,T]\times (\T^2)^N \big)$,
  \begin{equation}\label{equation-3}
  \int_0^T \!\! \int_{(\T^2)^N} (A_k\cdot \nabla u^a_t)(x) \beta(t,x) \,\d x\d t = - \int_0^T \!\! \int_{(\T^2)^N} u^a_t(x) (A_k\cdot \nabla \beta(t))(x) \,\d x\d t.
  \end{equation}
Applying \eqref{gradient-estimate-2.6} leads to
  \begin{equation}\label{gradient-estimate-6}
  \sum_{k=1}^N \int_0^T \big\|A_k\cdot \nabla u^a_t \big\|_{L^2}^2 \,\d t \leq \|v_a \|_\infty^2.
  \end{equation}
We remark that if $\tilde p(a)=0$ for some $a\in \R^N$, then $\rho(a,x) =0$ for all $x\in (\T^2)^N$ since $\rho$ is continuous. In this case it is natural to set $v_a(x) \equiv u^a_t(x) \equiv 0$, and the properties \eqref{equation-3} and \eqref{gradient-estimate-6} hold as well.

Now we compute the joint law of $\big(\xi, X^\xi_t \big)$ when the initial variables $(\xi, X_0)$ are distributed as $\rho(a,x) \lambda_N^0(\d a, \d x)$. For any bounded measurable function $F$ on $(\R\times \T^2)^N$,
  $$\aligned
  \E \big[ F\big(\xi, X^\xi_t\big)\big] &= \int_{\R^N} \E \big[ F\big(\xi, X^\xi_t\big) | \xi=a \big] \tilde p(a) \,\d a = \int_{\R^N} \E \big[ F(a, X^a_t ) \big] \tilde p(a) \,\d a \\
  &= \int_{\R^N} \int_{(\T^2)^N} F(a,x) u^a_t(x) \tilde p(a) \, \d x\d a.
  \endaligned$$
Thus, the joint distribution of $\big(\xi, X^\xi_t \big)$ is $u^a_t(x) \tilde p(a)\, \d x\d a$, and its density w.r.t. $\lambda_N^0$ is
  $$\tilde u_t(a,x)= u^a_t(x) \int_{(\T^2)^N} \rho(a,x)\,\d x.$$

Now we can transfer the property \eqref{equation-3} and the gradient estimate \eqref{gradient-estimate-6} to the density $\tilde u_t(a,x)$. First, for any $\beta \in C^{0,1}\big([0,T]\times (\T^2)^N \big)$, multiplying both sides of \eqref{equation-3} by $\int_{(\T^2)^N} \rho(a,x)\,\d x$ and integrating on $\R^N$ w.r.t. $p_N(a)\,\d a$ lead to
  \begin{equation}\label{equation-4}
  \aligned
  &\hskip13pt \int_0^T \!\! \int_{(\R\times \T^2)^N} \big(A_k(x)\cdot \nabla_{2N}\, \tilde u_t(a,x) \big) \beta(t,x) \,\lambda_N^0(\d a,\d x)\d t \\
  &= - \int_0^T \!\! \int_{(\R\times \T^2)^N} \tilde u_t(a,x) (A_k(x)\cdot \nabla_{2N} \beta(t,x)) \,\lambda_N^0(\d a,\d x)\d t,
  \endaligned
  \end{equation}
where $\nabla_{2N}$ is the gradient w.r.t. the $x$ variable. Next, multiplying both sides of \eqref{gradient-estimate-6} by $\big(\int_{(\T^2)^N} \rho(a,x)\,\d x \big)^2$, we obtain
  $$\aligned
  \sum_{k=1}^N \int_0^T \!\int_{(\T^2)^N} \big(A_k\cdot \nabla_{2N}\, \tilde u_t(a,\cdot) \big)^2 \, \d x\d t
  &\leq \bigg(\int_{(\T^2)^N} \rho(a,x)\,\d x\bigg)^2 \frac{\|\rho(a,\cdot)\|_\infty^2}{\big(\int_{(\T^2)^N} \rho(a,x)\,\d x \big)^2} \\
  &=  \|\rho(a,\cdot)\|_\infty^2,
  \endaligned$$
Integrating w.r.t. $p_N(a)\,\d a$ on $\R^N$ yields
  \begin{equation}\label{gradient-estimate-7}
  \sum_{k=1}^N \int_0^T \big\|A_k\cdot \nabla_{2N}\, \tilde u_t\big\|_{L^2(\lambda_N^0)}^2 \,\d t \leq  \|\rho\|_\infty^2.
  \end{equation}

\textbf{Transfer to gradient estimate on the density of vorticity.}

Given a bounded continuous function $\rho_0: \mathcal M_N(\T^2) \to \R_+$ such that $\int_{\mathcal M_N(\T^2)} \rho_0\, \d \mu_N^0 =1$. We consider the stochastic point vortex dynamics \eqref{stoch-point-vortex} with $(\xi, X_0)$ distributed as $(\rho_0\circ \mathcal T_N)(a,x) \lambda_N^0(\d a,\d x)$, where $\mathcal T_N$ is defined in \eqref{mapping}. Denoting again by $\tilde u_t(a,x)$ the joint density of $\big(\xi, X^\xi_t\big)$ w.r.t. $\lambda_N^0$, the gradient estimate \eqref{gradient-estimate-7} becomes
  \begin{equation}\label{gradient-estimate-8}
  \sum_{k=1}^N \int_0^T \big\|A_k\cdot \nabla_{2N}\, \tilde u_t\big\|_{L^2(\lambda_N^0)}^2 \,\d t \leq  \|\rho_0\|_\infty^2.
  \end{equation}

We intend to transform the above gradient estimate to the density function $\rho^N_t$ of
  $$\omega^N_t = \frac1{\sqrt N} \sum_{i=1}^N \xi_i \delta_{X^{\xi,i}_t}.$$
The existence of $\rho^N_t$ is due to Lemma \ref{lem-law-point-vortices}. We shall show that, for every $k\in S_N= \{1, \ldots, N\}$, $\big \<\sigma_k\cdot \nabla \omega, D_\omega \rho^N_t \big\>$ exists in the distributional sense, that is, there exists some $g_k\in L^2 \big([0,T]\times \mathcal M_N, \d t \otimes\mu_N^0 \big)$ such that for all $f\in \mathcal{FC}_{P,T}$,
  \begin{equation}\label{eq-5}
  \int_0^T\! \int_{\mathcal M_N} g_k(t,\omega) f(t,\omega)\,\mu_N^0(\d\omega)\d t = - \int_0^T\! \int_{\mathcal M_N} \rho^N_t(\omega) \big \<\sigma_k\cdot \nabla \omega, D_\omega f(t,\omega) \big\> \,\mu_N^0(\d\omega)\d t,
  \end{equation}
where $\mathcal M_N= \mathcal M_N(\T^2)$ and $D_\omega f(t,\omega)$ is defined before Theorem \ref{thm-main-result-2}.  $\mathcal{FC}_{P,T}$ is the family of test functionals defined in the introduction, which can also be regarded as smooth functionals on $\mathcal M_N$. To this end, we need the following simple facts.

\begin{lemma}\label{lem-gradient}
For any $G\in \mathcal{FC}_P$, under the map $(\R\times \T^2)^N \ni (a,x)\to \omega= \mathcal T_N(a, x) \in \mathcal M_N$,
  $$A_k(x) \cdot \nabla_{2N} (G\circ \mathcal T_N)(a, x)= - \big\<\sigma_k\cdot \nabla \omega, D_\omega G \big\>.$$
Moreover, $\div_{\mu_N^0} (\sigma_k\cdot \nabla \omega)=0$ in the sense of distribution; that is, for all $G\in \mathcal{FC}_P$,
  \begin{equation}\label{eq-divergence}
  \int_{\mathcal M_N} \big\<\sigma_k\cdot \nabla \omega, D_\omega G\big\>\,\d \mu_N^0=0.
  \end{equation}
\end{lemma}

\begin{proof}
Assume that $G\in \mathcal{FC}_P$ has the form $G(\omega)= g(\<\omega, \phi_1\>, \ldots, \<\omega, \phi_n\>)$; then
  $$(G\circ \mathcal T_N)(a, x)= g\bigg(\frac1{\sqrt N} \sum_{i=1}^N a_i \phi_1(x_i), \ldots, \frac1{\sqrt N} \sum_{i=1}^N a_i \phi_n(x_i) \bigg).$$
Recall the notation $\<\omega, \Phi\> = (\<\omega, \phi_1\>, \ldots, \<\omega, \phi_n\>)$ used in Section 4.1. By direct computation,
  \begin{align*}
  A_k(x) \cdot \nabla_{2N} (G\circ \mathcal T_N)(a, x)&= \sum_{i=1}^N \sigma_k(x_i) \cdot \partial_{x_i} (G\circ \mathcal T_N)(a, x) \\
  &= \sum_{i=1}^N \sigma_k(x_i) \cdot\bigg(\sum_{j=1}^n \partial_j g(\<\omega, \Phi\>) \frac{a_i}{\sqrt N} \nabla\phi_j(x_i)\bigg) \\
  &= \sum_{j=1}^n \partial_j g(\<\omega, \Phi\>) \<\omega, \sigma_k \cdot\nabla \phi_j\> = - \big\<\sigma_k\cdot \nabla \omega, D_\omega G \big\>.
  \end{align*}

Now we can prove the second assertion. We have
  $$\aligned \int_{\mathcal M_N} \big\<\sigma_k\cdot \nabla \omega, D_\omega G\big\>\,\mu_N^0(\d \omega)
  & = -\int_{(\R \times \T^2)^N} A_k(x) \cdot \nabla_{2N} (G\circ \mathcal T_N)(a, x) \, \lambda_N^0(\d a,\d x) \\
  & = -\int_{\R^N} p_N(a)\,\d a \int_{(\T^2)^N} A_k(x) \cdot \nabla_{2N} (G\circ \mathcal T_N)(a, x) \, \d x \\
  & = 0,
  \endaligned$$
since $\div_{2N} (A_k)(x)=0$. Here $p_N(a)$ is the density function of the standard Gaussian distribution on $\R^N$.
\end{proof}

Note that the law $\mu^N_t(\d\omega)= \rho^N_t(\omega) \mu_N^0(\d\omega)$ of $\omega^N_t$ is the image of that of $\big(\xi, X^\xi_t\big)$ under the map $\mathcal T_N: (\R\times \T^2)^N \to \mathcal M_N$. Therefore,
  \begin{equation*}
  \aligned
  &\int_0^T \! \int_{\mathcal M_N} \rho^N_t(\omega) \big \<\sigma_k\cdot \nabla \omega, D_\omega f(t) \big\> \,\mu_N^0(\d\omega)\d t\\
  = & \int_0^T \! \int_{(\R\times \T^2)^N} \big \<\sigma_k\cdot \nabla \omega, D_\omega f(t) \big\>|_{\omega= \mathcal T_N(a,x)}\, \tilde u_t(a,x)\, \lambda_N^0(\d a ,\d x) \d t\\
  = & - \int_0^T \! \int_{(\R\times \T^2)^N} \big[A_k(x) \cdot \nabla_{2N} \big(f(t)\circ \mathcal T_N \big)(a,x)\big]  \tilde u_t(a,x)\, \lambda_N^0(\d a ,\d x)\d t,
  \endaligned
  \end{equation*}
where in the last step we applied Lemma \ref{lem-gradient}. The integration by parts formula \eqref{equation-4} yields
  \begin{equation}\label{eq-6} \aligned
  &\int_0^T \! \int_{\mathcal M_N} \rho^N_t(\omega) \big \<\sigma_k\cdot \nabla \omega, D_\omega f(t) \big\> \,\mu_N^0(\d\omega)\d t\\
  = & \int_0^T \! \int_{(\R\times \T^2)^N} \big(f(t)\circ \mathcal T_N \big)(a,x)  \big[A_k(x) \cdot \nabla_{2N}\, \tilde u_t(a,x)\big]\, \lambda_N^0(\d a ,\d x) \d t\\
  = & \int_0^T \! \int_{(\R\times \T^2)^N} \big(f(t)\circ \mathcal T_N \big)(a,x)\,  \E\big(\big[A_k(x) \cdot \nabla_{2N}\, \tilde u_t(a,x)\big] \big| \mathcal G \big)\, \lambda_N^0(\d a ,\d x) \d t,
  \endaligned
  \end{equation}
where the conditional expectation is taken w.r.t. the probability measure $\lambda_N^0$, and $\mathcal G$ is the sub-$\sigma$-field of the Borel field of $(\R\times \T^2)^N$ defined as
  \begin{equation*}
  \mathcal G= \sigma\big( \big\{ F\circ \mathcal T_N \,|\, F: \mathcal M_N \to \R \mbox{ is measurable} \big\} \big).
  \end{equation*}
There exists some $g_k(t): \mathcal M_N\to \R$ such that
  \begin{equation}\label{eq-6.5}
  \E\big(\big[A_k(x) \cdot \nabla_{2N}\, \tilde u_t(a,x)\big] \big| \mathcal G \big) = -\big(g_k(t) \circ \mathcal T_N\big)(a,x)\quad \mbox{for all } k\in \{1,\ldots, N \}.
  \end{equation}
By the property of conditional expectation,
  $$\|g_k(t)\|_{L^2(\mu_N^0)} = \|g_k(t) \circ \mathcal T_N\|_{L^2(\lambda_N^0)} \leq \big\|A_k\cdot \nabla_{2N} \, \tilde u_t \big\|_{L^2(\lambda_N^0)}. $$
Combining this with the gradient estimate \eqref{gradient-estimate-8}, we have
  \begin{equation}\label{eq-7}
  \sum_{k=1}^N \int_0^T \|g_k(t)\|_{L^2(\mu_N^0)}^2\, \d t \leq \|\rho_0\|_\infty^2.
  \end{equation}

Substituting \eqref{eq-6.5} into \eqref{eq-6}, we obtain
  $$\aligned
  &\int_0^T \! \int_{\mathcal M_N} \rho^N_t(\omega) \big \<\sigma_k\cdot \nabla \omega, D_\omega f(t) \big\> \,\mu_N^0(\d\omega)\d t\\
  = & -\int_0^T \! \int_{(\R\times \T^2)^N} \big(f(t)\circ \mathcal T_N \big) (a,x) \big(g_k(t)\circ \mathcal T_N \big)(a,x)\, \lambda_N^0(\d a ,\d x)\d t\\
  = & -\int_0^T \! \int_{\mathcal M_N} g_k(t,\omega) f(t,\omega)\, \mu_N^0(\d\omega) \d t.
  \endaligned$$
This is the desired equality \eqref{eq-5}. Moreover, thanks to the fact $\div_{\mu_N^0}(\sigma_k\cdot \nabla \omega)=0$ proved in Lemma \ref{lem-gradient}, we conclude the existence of $\big\<\sigma_k \cdot\nabla \omega, D_\omega \rho^N_t \big\>$ in the distributional sense, and
  \begin{equation}\label{eq-7.5}
  \big\<\sigma_k \cdot\nabla \omega, D_\omega \rho^N_t \big\> = g_k(t,\omega),\quad k\in \{1,\ldots, N \}.
  \end{equation}
Combining this equality with \eqref{eq-7} yields the gradient estimate below:
  \begin{equation}\label{gradient-estimate-9}
  \sum_{k=1}^N \int_0^T \! \int_{\mathcal M_N} \big\<\sigma_k\cdot \nabla \omega, D_\omega \rho^N_t \big\>^2 \,\mu_N^0(\d\omega)\d t \leq \|\rho_0 \|_\infty^2.
  \end{equation}

\textbf{Step 3: Letting $N\to \infty$.} Now suppose that we are given $\rho_0 \in C_b\big( H^{-1-}(\T^2) , \R_+\big)$ such that $\int \rho_0 \,\d\mu=1$, where $\mu$ is the law of the white noise. Let $\rho_t$ be given in Theorem \ref{thm-main-result}. Our purpose in this step is to prove the gradient estimate \eqref{thm-main-result-2.2} on $\rho_t$.

For any $N\in \N$, consider the restriction $\rho^N_0$ of $\rho_0$ to $\mathcal M_N\ (\subset H^{-1-})$ and the stochastic point vortex dynamics starting from $C_N \rho^N_0(\omega) \mu_N^0(\d\omega)$, where $C_N$ is the normalizing constant: $C_N = \big( \int \rho^N_0 \,\d \mu_N^0\big)^{-1}$. Since $\mu_N^0$ converges weakly to $\mu$, we have $\lim_{N\to \infty} C_N =1$. By \eqref{gradient-estimate-9}, we know that the density $\rho^N_t$ of the stochastic point vortices $\omega^N_t$ satisfies the gradient estimate
  \begin{equation}\label{gradient-estimate-10}
  \sum_{k=1}^N \int_0^T \! \int_{\mathcal M_N} \big\<\sigma_k\cdot \nabla \omega, D_\omega \rho^N_t \big\>^2 \,\mu_N^0(\d\omega)\d t \leq C_N^2 \big\|\rho^N_0 \big\|_\infty^2 \leq C_N^2 \|\rho_0 \|_\infty^2,\quad N\in \N.
  \end{equation}

For every $k>N$, we define $\big\<\sigma_k\cdot \nabla \omega, D_\omega \rho^N_t \big\> =0$ for all $(t,\omega) \in [0,T] \times \mathcal M_N$. We want to show that the family $\big\{ \big\<\sigma_k\cdot \nabla \omega, D_\omega \rho^N_t \big\> \,|\, (k,t, \omega) \in \N\times [0,T] \times \mathcal M_N \big\}_{N\in \N}$ has a subsequence which converges in a certain sense to some $G\in L^2 \big(\N\times [0,T] \times H^{-1-}, \# \otimes \d t \otimes \mu\big)$, where $\#$ is the counting measure on $\N$. To this end, we denote by
  $$\aligned
  \mathcal{FC}_{P,T}(\N)= \big\{f: \N\times [0,T] \times H^{-1-} \to \R \, \big| \, &\exists \, n_f \in \N \mbox{ s.t. } f_k\in \mathcal{FC}_{P,T} \mbox{ for } k\leq n_f,\\
  & \mbox{and } f_k\equiv 0  \mbox{ for } k> n_f \big\}.
  \endaligned$$
It is a dense linear subspace of $L^2 \big(\N\times [0,T] \times H^{-1-}, \# \otimes \d t \otimes \mu\big)$. Fix an $f\in \mathcal{FC}_{P,T}(\N)$, by the definition of $\big\<\sigma_k\cdot \nabla \omega, D_\omega \rho^N_t \big\>$ (cf. \eqref{eq-5} and \eqref{eq-7.5}), we have, for all $N > n_f$,
  \begin{equation}\label{eq-13}
  \aligned &\sum_{k=1}^\infty \int_0^T \!\int_{\mathcal M_N} \big\<\sigma_k\cdot \nabla \omega, D_\omega \rho^N_t \big\> f_k(t,\omega)\,\mu_N^0(\d\omega) \d t\\
  = & - \sum_{k=1}^\infty \int_0^T \!\int_{\mathcal M_N} \rho^N_t(\omega) \big\<\sigma_k\cdot \nabla \omega, D_\omega f_k(t,\omega)\big\>\,\mu_N^0(\d\omega) \d t.
  \endaligned
  \end{equation}
Note that the sums over $k\in \N$ on both sides are indeed finite sums. To proceed further, we need some preparations.

\begin{lemma}\label{lem-weak-convergence}
\begin{itemize}
\item[\rm (1)] Let $\{N_i\}_{i\in \N}$ be the subsequence obtained before Lemma \ref{lem-absolute-continuity}; then for any $F\in \mathcal{FC}_{P,T}$,
  $$\lim_{i\to \infty} \int_0^T \!\int_{\mathcal M_{N_i}} \rho^{N_i}_t(\omega) F(t, \omega) \,\mu_{N_i}^0(\d\omega) \d t= \int_0^T \!\int_{H^{-1-}} \rho_t(\omega) F(t, \omega) \,\mu(\d\omega) \d t.$$
\item[\rm (2)] For any  $G\in \mathcal{FC}_{P}$,
  $$\lim_{N\to \infty} \int_{\mathcal M_N} G(\omega) \,\mu_N^0(\d\omega) =  \int_{H^{-1-}} G(\omega) \,\mu(\d\omega).$$
\end{itemize}
\end{lemma}

\begin{proof}
(1) It suffices to prove the limit for $F(t, \omega)= f(t)G(\omega)$, where $f\in C([0,T])$ and $G(\omega)= g(\<\omega, \phi_1\>, \ldots, \<\omega, \phi_n\>) \in \mathcal{FC}_P$. We have
  $$\int_0^T \!\int_{\mathcal M_{N_i}} \rho^{N_i}_t(\omega) F(t,\omega) \,\mu_{N_i}^0(\d\omega) \d t = \int_0^T f(t) \hat\E \big[ G\big( \hat\omega^{N_i}_t \big)\big] \,\d t,$$
where $\hat\E$ is the expectation on the probability space $\big(\hat\Theta, \hat{\mathcal F},\hat\P\big)$, which comes from the Skorokhod's representation theorem in Section 3. If $G$ is bounded, by \eqref{lem-absolute-continuity.1} and the dominated convergence theorem, we see that the limit holds true. Using the method of truncation, it is sufficient to show that $\big\{G\big( \hat\omega^{N_i}_t \big) \big\}_{i\in \N}$ is bounded in $L^2\big([0,T]\times \hat\Theta \big)$. By Lemma \ref{lem-law-point-vortices},
  $$\hat\E\big[ G^2 \big( \hat\omega^{N_i}_t \big)\big]= \int_{\mathcal M_{N_i}} G^2(\omega) \rho^{N_i}_t(\omega) \,\mu_{N_i}^0(\d\omega) \leq C_{N_i} \|\rho_0\|_\infty  \int_{\mathcal M_{N_i}} G^2(\omega) \,\mu_{N_i}^0(\d\omega).$$
Note that $G(\omega)= g(\<\omega, \phi_1\>, \ldots, \<\omega, \phi_n\>)$ and $g$ has polynomial growth. Combining this fact with the definition of $\mu_{N}^0$ in Section 2.1, some simple calculations lead to the desired result.

(2) The proof is similar as above; the only difference is that we replace the limit \eqref{lem-absolute-continuity.1} by the weak convergence of $\mu_N^0$ to $\mu$ proved in Proposition \ref{prop-weak-convergence}.
\end{proof}

By Remark \ref{1-rem}(1), we have $\sum_{k=1}^\infty \big\<\sigma_k\cdot \nabla \omega, D_\omega f_k(t,\omega)\big\> = \sum_{k=1}^{n_f} \big\<\sigma_k\cdot \nabla \omega, D_\omega f_k(t,\omega)\big\>\in \mathcal{FC}_{P,T}$. The first assertion of Lemma \ref{lem-weak-convergence} leads to
  \begin{equation}\label{eq-14} \aligned
  \lim_{i\to \infty} &\sum_{k=1}^\infty \int_0^T \!\int_{\mathcal M_{N_i}} \rho^{N_i}_t(\omega) \big\<\sigma_k\cdot \nabla \omega, D_\omega f_k(t,\omega)\big\>\,\mu_{N_i}^0(\d\omega) \d t\\
  = &\sum_{k=1}^\infty \int_0^T \!\int_{H^{-1-}} \rho_t(\omega) \big\<\sigma_k\cdot \nabla \omega, D_\omega f_k(t,\omega) \big\>\, \mu(\d\omega) \d t.
  \endaligned
  \end{equation}
From \eqref{eq-13} and \eqref{eq-14} we see that, for all $f\in \mathcal{FC}_{P,T}(\N)$, the limit
  $$\lim_{i\to \infty} \sum_{k=1}^\infty \int_0^T \!\int_{\mathcal M_{N_i}} \big\<\sigma_k\cdot \nabla \omega, D_\omega \rho^{N_i}_t (\omega) \big\> f_k(t,\omega) \,\mu_{N_i}^0(\d\omega) \d t$$
exists, and denoting it by $L(f)$, we have
  \begin{equation}\label{eq-15}
  L(f)= -\sum_{k=1}^\infty \int_0^T \!\int_{H^{-1-}} \rho_t(\omega) \big\<\sigma_k\cdot \nabla \omega, D_\omega f_k(t,\omega) \big\>\, \mu(\d\omega) \d t.
  \end{equation}
The above equality clearly implies that $L$ is an linear functional on $\mathcal{FC}_{P,T}(\N)$, which is a dense subspace of $L^2 \big(\N\times [0,T] \times H^{-1-}, \# \otimes \d t \otimes \mu\big)$. Moreover, by Cauchy's inequality,
  $$\aligned |L(f)| &= \lim_{i\to \infty} \bigg| \sum_{k=1}^\infty \int_0^T \!\int_{\mathcal M_{N_i}} \big\<\sigma_k\cdot \nabla \omega, D_\omega \rho^{N_i}_t (\omega) \big\> f_k(t,\omega) \,\mu_{N_i}^0(\d\omega) \d t \bigg| \\
  &\leq \liminf_{i\to \infty} \big\|\big\<\sigma_k\cdot \nabla \omega, D_\omega \rho^{N_i}_t \big\> \big\|_{L^2(\N\times [0,T] \times \mathcal M_{N_i})}\, \|f\|_{L^2(\N\times [0,T] \times \mathcal M_{N_i})} \\
  &\leq \|\rho_0\|_\infty \|f\|_{L^2(\N\times [0,T] \times H^{-1-})},
  \endaligned$$
where in the last step we have used \eqref{gradient-estimate-10} and the second assertion of Lemma \ref{lem-weak-convergence}, by regarding $\sum_{k=1}^\infty f_k^2(t,\omega) = \sum_{k=1}^{n_f} f_k^2(t,\omega)\in \mathcal{FC}_{P,T}$ as a test function.

Summarizing the above arguments, we see that $L: \mathcal{FC}_{P,T}(\N) \to \R$ is a bounded linear functional, and thus it can be extended to the whole space $L^2 \big(\N\times [0,T] \times H^{-1-}, \# \otimes \d t \otimes \mu\big)$ as a bounded linear functional with the norm $\|L\| \leq \|\rho_0\|_\infty$. By Riesz's representation theorem, there exists $G\in L^2 \big(\N\times [0,T] \times H^{-1-}, \# \otimes \d t \otimes \mu\big)$ such that
  \begin{equation}\label{eq-16}
  \|G\|_{L^2(\N\times [0,T] \times H^{-1-})} = \|L\| \leq \|\rho_0\|_\infty,
  \end{equation}
and for all $f\in L^2 \big(\N\times [0,T] \times H^{-1-}, \# \otimes \d t \otimes \mu\big)$,
  $$L(f)= \sum_{k=1}^\infty \int_0^T \!\int_{H^{-1-}} G_k(t,\omega) f_k(t,\omega) \, \mu(\d\omega) \d t.$$
In particular, for every fixed $k\in \N$, taking $f\in \mathcal{FC}_{P,T}(\N)$ such that $f_j\equiv 0$ for all $j\neq k$, we have by \eqref{eq-15} that
  \begin{equation}\label{eq-17}
  \int_0^T \!\int_{H^{-1-}} G_k(t,\omega) f_k(t,\omega) \, \mu(\d\omega) \d t = - \int_0^T \!\int_{H^{-1-}} \rho_t(\omega) \big\<\sigma_k\cdot \nabla \omega, D_\omega f_k(t,\omega) \big\>\, \mu(\d\omega) \d t.
  \end{equation}
We need the final preparation.

\begin{lemma}\label{lem-divergence}
For all $k\in \N$, $\div_\mu (\sigma_k\cdot \nabla\omega)=0$ in the distributional sense, i.e. for all $G\in \mathcal{FC}_P$,
  $$\int_{H^{-1-}} \<\sigma_k\cdot \nabla\omega, D_\omega G\>\,\mu(\d\omega)=0.$$
\end{lemma}

\begin{proof}
By (1) of Remark \ref{1-rem}, we have $\<\sigma_k\cdot \nabla\omega, D_\omega G\> \in \mathcal{FC}_P$. The desired result follows from \eqref{eq-divergence} and the second assertion of Lemma \ref{lem-weak-convergence}.
\end{proof}

Combining Lemma \ref{lem-divergence} and \eqref{eq-17}, we see that $\big\<\sigma_k\cdot \nabla \omega, D_\omega \rho_t(\omega) \big\>$ exists in the distributional sense and
  $$\big\<\sigma_k\cdot \nabla \omega, D_\omega \rho_t(\omega) \big\> = G_k(t,\omega),\quad k\in \N.$$
Substituting this equality into \eqref{eq-16} eventually leads to the gradient estimate
  $$\sum_{k=1}^\infty \int_0^T \! \int_{H^{-1-}} \big\<\sigma_k\cdot \nabla \omega, D_\omega \rho_t(\omega) \big\>^2 \,\mu(\d\omega) \d t \leq \|\rho_0 \|_\infty^2.$$

\section{$L^2$-integrability of $\<b(\omega), D_\omega \rho_t\>$}

Our purpose is to prove Theorem \ref{thm-integrability} for which we need some preparations. Recall that $\big\<\sigma_k \cdot \nabla \omega, D_\omega \rho_t \big\>$ is characterized by the following identity: for all $F\in \mathcal{FC}_{P,T}$,
  $$\int_0^T \!\! \int_{H^{-1-}}\big\<\sigma_k \cdot \nabla \omega, D_\omega \rho_t \big\> F(t,\omega) \, \d\mu \d t = - \int_0^T \!\! \int_{H^{-1-}} \rho_t(\omega) \big\<\sigma_k\cdot \nabla \omega, D_\omega F(t,\omega) \big\>\, \d\mu \d t.$$
Taking $F(t,\omega)= f(t) G(\omega)$ with $f\in C([0,T])$ and $G\in \mathcal{FC}_{P}$, we deduce that
  $$\int_{H^{-1-}}\big\<\sigma_k \cdot \nabla \omega, D_\omega \rho_t \big\> G(\omega) \, \d\mu = - \int_{H^{-1-}} \rho_t(\omega) \big\<\sigma_k\cdot \nabla \omega, D_\omega G \big\>\, \d\mu \quad \mbox{for a.e. } t\in [0,T].$$
Now for $G(\omega)= \big\<\omega, {\rm e}^{2\pi {\rm i} k \cdot x} \big\> \in \mathcal{FC}_P$, we have $D_\omega G= {\rm e}^{2\pi {\rm i} k \cdot x}$. By the above equality, for a.e. $t\in [0,T]$,
  \begin{equation}\label{eq-10}
  \aligned
  \int_{H^{-1-}} \big\<\omega, {\rm e}^{2\pi {\rm i} k \cdot x} \big\> \big\<\sigma_k \cdot \nabla \omega, D_\omega \rho_t \big\> \,\d\mu &= - \int_{H^{-1-}} \big\<\sigma_k \cdot \nabla \omega, {\rm e}^{2\pi {\rm i} k \cdot x} \big\> \rho_t(\omega) \,\d\mu \\
  &= \int_{H^{-1-}} \big\<\omega, \sigma_k \cdot \nabla {\rm e}^{2\pi {\rm i} k \cdot x} \big\> \rho_t(\omega) \,\d\mu\\
  &= 0,
  \endaligned
  \end{equation}
since
  $$\sigma_k \cdot \nabla {\rm e}^{2\pi {\rm i} k \cdot x} = {\rm e}^{2\pi {\rm i} k \cdot x} \frac{k^\perp}{|k|^2} \cdot {\rm e}^{2\pi {\rm i} k \cdot x} 2\pi {\rm i} k =0.$$

For simplicity of notations, we denote by
  \begin{equation}\label{eq-rv}
  \xi_k(t, \omega) = \big\<\sigma_k \cdot \nabla \omega, D_\omega \rho_t \big\>,\quad \eta_k(\omega) = \big\<\omega, {\rm e}^{2\pi {\rm i} k \cdot x} \big\>, \quad k\in \Z^2_0= \Z^2 \setminus \{0\}.
  \end{equation}
We summarize the properties of $\xi_k(t)$ and $\eta_k$ for later use.

\begin{lemma}\label{lem-properties}
\begin{itemize}
\item[\rm(i)] $ \sum_{k\in \Z^2_0} \int_0^T \|\xi_k(t)\|_{L^2(\mu)}^2 \,\d t<+\infty$;
\item[\rm(ii)] by \eqref{eq-10}, for any $k\in \Z^2_0$, $\<\xi_k(t), \eta_k\>_{L^2(\mu)} =0$ for a.e. $t\in [0,T]$;
\item[\rm(iii)] under $\mu$, the family $\{\eta_k\}_{k\in \Z^2_0}$ consists of i.i.d. complex-valued standard Gaussian r.v.'s;
\item[\rm(iv)] the family $\big\{\eta_k(\omega)= \big\<\omega, {\rm e}^{2\pi {\rm i} k \cdot x} \big\> \big\}_{k\in \Z^2_0 }$ is an orthonormal basis of
    $$L^2_0(H^{-1-}, \mu)= \big\{F\in L^2(H^{-1-}, \mu): \E_\mu F=0 \big\}. $$
\end{itemize}
\end{lemma}

Theorem \ref{thm-integrability} is a consequence of the following result.

\begin{proposition}\label{prop-convergence}
The series $\sum_{k\in \Z^2_0 } \xi_k(t, \omega)\, \eta_k(\omega)$ converges in $L^2([0,T]\times H^{-1-})$.
\end{proposition}

\begin{proof}
Denote by
  $$J_N(t,\omega):= \sum_{0<|k|\leq N} \xi_k(t,\omega) \, \eta_k(\omega),\quad N\in \N.$$
It is sufficient to show that $\{J_N\}_{N\in \N}$ is a Cauchy sequence in  $L^2([0,T]\times H^{-1-})$.

By (i) and (iv) in Lemma \ref{lem-properties}, we have the orthogonal decomposition: for any $k\in \Z^2_0 $,
  \begin{equation}\label{decomposition}
  \xi_k(t) = \sum_{l\in \Z^2_0} \<\xi_k(t) , \eta_l \>_{L^2(\mu)} \, \eta_l = \sum_{0< |l|\leq n} \<\xi_k(t) , \eta_l \>_{L^2(\mu)} \, \eta_l + \xi_{k,n}(t),
  \end{equation}
where $n\in \N$ is any fixed integer and $\xi_{k,n}(t)$ denotes the remainder part.  For $N, M\in \N, N<M$, we have
  \begin{equation}\label{convergence-1}
  \aligned
  J_M- J_N &=\sum_{N< |k|\leq M} \xi_k(t)\, \eta_k = \sum_{N< |k|\leq M} \eta_k \bigg(\sum_{0< |l|\leq n} \<\xi_k(t) , \eta_l \>_{L^2(\mu)} \, \eta_l + \xi_{k,n}(t) \bigg)\\
  &= \sum_{N< |k|\leq M} \sum_{0< |l|\leq n} \<\xi_k(t) , \eta_l \>_{L^2(\mu)} \, \eta_k \, \eta_l + \sum_{N< |k|\leq M} \eta_k \, \xi_{k,n}(t) \\
  &=: I_1(n) + I_2(n).
  \endaligned
  \end{equation}

\begin{lemma}\label{5-lem-1}
We have
  $$ \sup_{n\geq M} \int_0^T \|I_1(n) \|_{L^2(\mu)}^2 \,\d t \leq 2 \sum_{N< |k|\leq M} \int_0^T \|\xi_k(t)\|_{L^2(\mu)}^2 \,\d t.$$
\end{lemma}

\begin{proof}
Fix any $n\geq M$. Note that
  $$|I_1(n)|^2= \sum_{N< |k|,|k'|\leq M} \sum_{0< |l|, |l'|\leq n} \<\xi_k(t) , \eta_l \>_{L^2(\mu)} \overline{\<\xi_{k'}(t) , \eta_{l'} \>_{L^2(\mu)}} \, \eta_k\, \eta_l \, \overline{\eta_{k'}}\, \overline{\eta_{l'}}. $$
In view of property (ii) in Lemma \ref{lem-properties}, the terms with $k=l$ or $k' =l'$ vanish. Recall that $\{\eta_k\}_{k\in \Z^2_0}$ is a family of i.i.d. standard Gaussian r.v.'s. Similar to the proof of \cite[Lemma 23, Step 3]{Flandoli}, we consider the following cases:

(1) $k=k'\neq l = l'$. The sum of these finite terms is denoted by $\hat J$, then
  \begin{equation*}
  \aligned
  \int_0^T \E_\mu \big(\hat J \big)\,\d t &= \sum_{N< |k|\leq M} \sum_{0< |l|\leq n,\, l\neq k} \int_0^T \big|\<\xi_k(t) , \eta_l \>_{L^2(\mu)} \big|^2\, \E_\mu \big(|\eta_k|^2\, |\eta_l|^2\big) \,\d t \\
  &= \sum_{N< |k|\leq M} \sum_{0< |l|\leq n,\, l\neq k} \int_0^T \big|\<\xi_k(t) , \eta_l \>_{L^2(\mu)} \big|^2 \, \E_\mu \big(|\eta_k|^2 \big) \E_\mu \big(|\eta_l|^2 \big) \,\d t \\
  &\leq \sum_{N< |k|\leq M} \int_0^T \|\xi_k(t)\|_{L^2(\mu)}^2 \,\d t,
  \endaligned
  \end{equation*}
where in the third equality we used $\E_\mu \big(|\eta_k|^2 \big) =1$ and \eqref{decomposition}. \medskip

(2) $k=l \neq k' = l'$. In this case, by property (ii) in Lemma \ref{lem-properties}, $\<\xi_k(t) , \eta_l \>_{L^2(\mu)} =\<\xi_{k'}(t) , \eta_{l'} \>_{L^2(\mu)} =0$ for a.e. $t\in [0,T]$.

The case $k=l = k' = l'$ is treated analogously. \medskip

(3) $k=l' \neq k' = l$. The sum of these terms is denoted by $\tilde J$. We have
  \begin{equation*}
  \aligned
  \int_0^T \E_\mu \big(\tilde J \big)\,\d t &= \sum_{N< |k|, |k'|\leq M} \int_0^T \<\xi_k(t) , \eta_{k'} \>_{L^2(\mu)} \overline{ \<\xi_{k'}(t) , \eta_k \>_{L^2(\mu)}}\, \E_\mu \big(|\eta_k|^2\, |\eta_{k'}|^2\big) \,\d t \\
  &= \sum_{N< |k|, |k'|\leq M}  \int_0^T \<\xi_k(t) , \eta_{k'} \>_{L^2(\mu)} \overline{\<\xi_{k'}(t) , \eta_k \>_{L^2(\mu)}}\, \,\d t.
  \endaligned
  \end{equation*}
Changing $k$ to $k'$ and $k'$ to $k$, we see that the r.h.s. is real. By Cauchy's inequality,
  \begin{equation*}
  \aligned
  \int_0^T \E_\mu \big(\tilde J \big)\,\d t &\leq \bigg[\sum_{N< |k|, |k'|\leq M} \int_0^T \! \big|\<\xi_k(t) , \eta_{k'} \>_{L^2(\mu)} \big|^2 \d t \bigg]^{\frac12} \bigg[\sum_{N< |k|, |k'|\leq M} \int_0^T \! \big| \<\xi_{k'}(t) , \eta_k \>_{L^2(\mu)} \big|^2 \d t \bigg]^{\frac12} \\
  &\leq \bigg[\sum_{N< |k|\leq M} \int_0^T \|\xi_k(t)\|_{L^2(\mu)}^2 \,\d t \bigg]^{\frac12} \bigg[\sum_{N< |k'|\leq M} \int_0^T \|\xi_{k'}(t)\|_{L^2(\mu)}^2 \,\d t \bigg]^{\frac12} \\
  &= \sum_{N< |k|\leq M} \int_0^T \|\xi_k(t)\|_{L^2(\mu)}^2 \,\d t.
  \endaligned
  \end{equation*}

Summarizing the above arguments yields the desired estimate.
\end{proof}

Next we consider the quantity $I_2(n)$.

\begin{lemma}\label{lem-2}
For fixed $M >N$, we have
  $$\lim_{n\to \infty} \int_0^T \|I_2(n)\|_{L^2(\mu)}^2 \,\d t =0.$$
\end{lemma}

\begin{proof}
There exists $C= C(M, N)>0$ such that
  $$|I_2(n)|^2 \leq C \sum_{N< |k|\leq M} |\eta_k|^2 \, |\xi_{k,n}(t)|^2.$$
First, we show that each term on the r.h.s. belongs to $L^1(H^{-1-}, \mu)$. Since $n>M$, by property (iii) in Lemma \ref{lem-properties} and the definition \eqref{decomposition} of $\xi_{k,n}(t)$, we see that $\eta_k$ and $\xi_{k,n}(t)$ are independent. Therefore, for any $f,g\in C_b(\R)$,
  $$\E_\mu \big[ f\big(|\eta_k|^2 \big) g\big(|\xi_{k,n}(t)|^2 \big) \big]= \E_\mu \big[ f\big(|\eta_k|^2 \big) \big] \E_\mu \big[ g\big(|\xi_{k,n}(t)|^2 \big) \big].$$
Taking $f(x)=g(x)= x\wedge R,\, x\in [0,\infty), R>0$, and by the monotone convergence theorem, we obtain
  $$\E_\mu \big[|\eta_k|^2|\xi_{k,n}(t)|^2 \big]= \E_\mu \big[|\eta_k|^2\big] \E_\mu \big[|\xi_{k,n}(t)|^2 \big]= \E_\mu \big[|\xi_{k,n}(t)|^2 \big] = \|\xi_{k,n}(t)\|_{L^2(\mu)}^2.$$
Therefore, all the terms are integrable. As a result,
  $$ \int_0^T \|I_2(n)\|_{L^2(\mu)}^2\,\d t \leq C \sum_{N< |k|\leq M} \int_0^T \|\xi_{k,n}(t)\|_{L^2(\mu)}^2 \,\d t. $$
For every $k\in \Z^2\setminus \{0\}$, since $\xi_{k,n}(t) $ is the remainder in the decomposition \eqref{decomposition}, we have
  $$\lim_{n\to \infty} \int_0^T \|\xi_{k,n}(t)\|_{L^2(\mu)}^2 \,\d t =0.$$
This implies the desired limit.
\end{proof}

Combining \eqref{convergence-1} and Lemma \ref{5-lem-1}, we have
  $$\aligned
  \int_0^T \|J_M -J_N\|_{L^2(\mu)}^2 \,\d t & \leq  2 \int_0^T \|I_1(n)\|_{L^2(\mu)}^2 \,\d t + 2 \int_0^T  \|I_2(n)\|_{L^2(\mu)}^2 \,\d t \\
  & \leq 4 \sum_{N< |k|\leq M} \int_0^T \|\xi_k(t)\|_{L^2}^2 \,\d t  + 2\int_0^T \|I_2(n)\|_{L^2(\mu)}^2\,\d t .
  \endaligned$$
By Lemma \ref{lem-2} and property (i) of Lemma \ref{lem-properties}, first letting $n\to \infty$ and then $N,M \to \infty$, we see that $\{J_N\}_{N\in \N}$ is a Cauchy sequence in $L^2([0,T]\times H^{-1-})$. The proof of Proposition \ref{prop-convergence} is complete.
\end{proof}

\medskip

\noindent \textbf{Acknowledgement.} The second author is grateful to the financial supports of the National Natural Science Foundation of China (Nos. 11431014, 11571347), the Seven Main Directions (Y129161ZZ1) and the Special Talent Program of the Academy of Mathematics and Systems Science, Chinese Academy of Sciences.

\end{document}